\documentclass{article}

\usepackage[utf8]{inputenc}

\usepackage{parskip, amsfonts, amsmath, changepage, amsthm, setspace, geometry, verbatim, MnSymbol,chemarrow}

\usepackage[pdftex]{graphicx}
\graphicspath{ {./images/} }

\usepackage[T1]{fontenc}

\usepackage{authblk, filecontents}

\usepackage{MnSymbol, wasysym, mathtools,xr}

\usepackage[tiny]{titlesec}
\usepackage{hyperref}
\newcommand*\Mod[1]{ \; (\textup{mod} \; #1 )}
\newcommand*\tops[2]{\texorpdfstring{#1}{#2}}
\newcommand*{\Z}{\mathbb{Z}}

\renewcommand{\phi}{\varphi}
\newtheorem{Thm}{Theorem}[section]

\newtheorem{Lem}[Thm]{Lemma}
\newtheorem{Cor}[Thm]{Corollary}

\theoremstyle{definition}
\newtheorem{Def}[Thm]{Definition}

\theoremstyle{remark}

\providecommand{\keywords}[1]
{
  \small	
  \textbf{\textit{Keywords---}} #1
}

\title{Chromatic numbers of Cayley graphs of abelian groups: Cases of small dimension and rank}
\author[a]{Jonathan Cervantes}
\author[b]{Mike Krebs}
\affil[a]{University of California, Riverside, Dept. of Mathematics, Skye Hall, 900 University Ave., Riverside, CA 92521, jcerv092@ucr.edu}
\affil[b]{California State University, Los Angeles, Dept. of Mathematics, 5151 State University Drive, Los Angeles, CA 91711, mkrebs@calstatela.edu}
\date{\today}

\begin{document}

\maketitle


\begin{abstract}
A connected Cayley graph on an abelian group with a finite generating set $S$ can be represented by its Heuberger matrix, i.e., an integer matrix whose columns generate the group of relations between members of $S$.  In a previous article, the authors laid the foundation for the use of Heuberger matrices to study chromatic numbers of abelian Cayley graphs.  We call the number of rows in the Heuberger matrix the {\it dimension}, and the number of columns the {\it rank}.  In this paper,  we give precise numerical conditions that completely determine the chromatic number in all cases with dimension $1$; with rank $1$; and with dimension $\leq 3$ and rank $\leq 2$.  For such a graph without loops, we show that it is $4$-colorable if and only if it does not contain a $5$-clique, and it is $3$-colorable if and only if it contains neither a diamond lanyard nor a $C_{13}(1,5)$, both of which we define herein.  It is shown in another paper that as a special case of our theorem for dimension $3$ and rank $2$, we obtain improved upper bounds for minimal periods of optimal colorings of $6$-valent integer distance graphs.
\end{abstract}

\section{Introduction}

%%%%%%%%% %%%%%%%%% %%%%%%%%% 
%%%%%%%%% This will be edited by:
% KREBS
%%%%%%%%% 
%%%%%%%%% %%%%%%%%% %%%%%%%%% 

Recall that a subset $S$ of a group $G$ is said to be \emph{symmetric} if $x^{-1}\in S$ whenever $x\in S$.  Given a group $G$ and a symmetric subset $S$ of $G$, the \emph{Cayley graph} $\text{Cay}(G,S)$ is defined to be the graph whose vertex set is $G$, where two vertices $x$ and $y$ are adjacent if and only if $x^{-1}y\in S$.

Given an $m\times r$ integer matrix $M$, let $H$ be the set of all linear combinations of the columns of $M$ with integer coefficients.  Let $\mathbb{Z}$ denote the group of integers under addition, and let $\mathbb{Z}^m$ denote the $m$-fold direct product of  $\mathbb{Z}$ with itself.  Let $e_j\in\mathbb{Z}^m$ denote the $m$-tuple (regarded as a column vector) whose $i$th component is $1$ if $i=j$ and is $0$ otherwise.  Let $S=\{H\pm e_1,\dots,H\pm e_m\}$.  We may then form the Cayley graph whose underlying group is $\mathbb{Z}^m/H$ with respect to the generating set $S$.  We denote by $M^{\text{SACG}}_X$ the graph formed in this manner.  We call $M^{\text{SACG}}_X$ a {\it standardized abelian Cayley graph}, and we call $M$ an associated {\it Heuberger matrix}.  As discussed in \cite{Cervantes-Krebs-general-publilshed}, the study of chromatic numbers of Cayley graphs on abelian groups can be reduced to the study of standardized abelian Cayley graphs and their Heuberger matrices.  Many, many particular cases of 
chromatic numbers of Cayley graphs on abelian groups have been studied; see the introduction to \cite{Cervantes-Krebs-general-publilshed} for a long list of examples.

Our main results (Theorems \ref{theorem-m-equals-2} and \ref{theorem-m-equals-3}) give precise and easily checked numerical conditions that completely determine the chromatic number when the associated Heuberger matrix is $2\times 2$ or $3\times 2$.  These results can be summarized as follows:  Suppose such a graph does not have loops.  Then it is $4$-colorable if and only if it does not contain a $5$-clique, and it is $3$-colorable if and only if it contains neither a diamond lanyard (see Def. \ref{def-diamond-lanyard}) nor a $\text{Cay}(\mathbb{Z}_{13},\{\pm 1,\pm 5\})$.

Whether such subgraphs occur, we show, can be ascertained quickly from the entries of the Heuberger matrix.  After excluding some trivial exceptional cases, one first puts the matrix into a certain standardized form (lower triangular form with positive diagonal entries for $2 \times 2$ matrices, ``modified Hermite normal form'' for $3\times 2$ matrices) without changing the associated graph.  Theorems \ref{theorem-m-equals-2} and \ref{theorem-m-equals-3} then provide formulas for the chromatic number of graphs with matrices in this form.

We briefly sketch the method of proof.  For $2\times 2$ matrices, the main result (Thm. \ref{theorem-m-equals-2}) follows quickly by combining Heuberger's theorem on chromatic numbers of circulant graphs with the methods of \cite{Cervantes-Krebs-general-publilshed}.  The principal result for $3\times 2$ matrices (Thm. \ref{theorem-m-equals-3}) requires more work.  The central idea is to apply the graph homomorphisms of \cite{Cervantes-Krebs-general-publilshed} to obtain upper bounds on the chromatic number, utilizing the previous results from the $2\times 2$ case.

We remark that both Thm. \ref{theorem-m-equals-2} and Thm. \ref{theorem-m-equals-3} relate to the so-called ``shortest vector problem'' from lattice theory.  It follows directly from Thm. \ref{theorem-m-equals-3} (respectively, Thm. \ref{theorem-m-equals-2}) that if the minimal $L^1$ norm of a vector in the $\mathbb{Z}$-span of a $3\times 2$ (resp. $2\times 2$) integer matrix without zero rows is greater than $3$ (resp. $6$), then the corresponding graph is $3$-colorable.  

In \cite{Zhu}, Zhu finds the chromatic number for an arbitrary integer distance graph of the form $\text{Cay}(\mathbb{Z},\{\pm a,\pm b,\pm c\})$, where $a, b,$ and $c$ are distinct positive integers.  Such graphs, as shown in \cite{Cervantes-Krebs-general-publilshed}, have associated $3\times 2$ matrices.  In another companion paper \cite{Cervantes-Krebs-Zhu}, we demonstrate how Thm. \ref{theorem-m-equals-3} yields Zhu's theorem as a corollary of our main theorem about $3\times 2$ Heuberger matrices.  Indeed, we moreover show there that this method gives us vastly improved upper bounds for minimal periods of optimal colorings of such graphs.

One obvious future direction for this work will be to investigate what happens when the matrices are larger.  For example, the case of a graph $X$ with an associated $m\times 2$ Heuberger matrix when $m\geq 4$ seems well within reach using methods developed in this paper, and we plan to tackle that next.

As we show in the proofs of our main theorems, when an abelian Cayley graph $X$ has an associated $2\times 2$ or $3\times 2$ Heuberger matrix, an optimal coloring for $X$ can always be realized as a pullback, via a graph homomorphism, of a coloring of a circulant graph.  We pose the question (akin to that asked in \cite{MathOverflow}): Is that statement true for all connected, finite-degree abelian Cayley graphs?

Moreover, we propose the following conjecture: Let $X$ be a standardized abeliam Cayley graph with an associated Heuberger matrix $M_X$.  If $X$ does not have loops, and the rank over $\mathbb{Z}_3$ of the reduction of $M_X$ modulo $3$ is $\leq 1$, then $X$ is $3$-colorable.

This article depends heavily on \cite{Cervantes-Krebs-general-publilshed}, which we will refer to frequently.  The reader should assume that all notation, terminology, and theorems used but not explained here are explained there.

\section{Main theorems}\label{section-main-theorems}

In \cite{Cervantes-Krebs-general-publilshed}, we laid the groundwork for our main techniques, and we employed them to prove the Tomato Cage Theorem, which completely determines the chromatic number in the case where $H$ has rank $1$.  In this section, we turn our attention to our main results, which concern the case where $H$ has rank $2$.  Let $X$ be a standardized abelian Cayley graph, and let $m$ be the number of rows in an associated Heuberger matrix $M_X$.

In Subsection \ref{subsection-m-is-1}, we quickly dispense with the case where $m=1$.

% In Subsection \ref{subsection-diamond-lanyards}, we define and discuss diamond lanyards and $C_{13}(1,5)$.  As we shall see, these are the only obstructions to $3$-colorability when the Heuberger matrix is of size $2\times 2$ or $3\times 2$, assuming no loops.

For $m=2$, we first apply isomorphisms to $X$ as in \cite{Cervantes-Krebs-general-publilshed} to put the  matrix in a standard form without changing the associated graph.  We then clear aside the somewhat aberrant situations where $X$ is bipartite, $X$ has loops, or the matrix has a zero row.  Excluding these possibilities, we show that if the matrix entries are not relatively prime, then $\chi(X)=3$; otherwise, $X$ is isomorphic to a circulant graph, and its chromatic number is given by a theorem of Heuberger's.  Subsection \ref{subsection-m-2} contains the precise statements and proofs.

For $m=3$, we again begin by putting the matrix into a standard form.  In this case it is a certain form we call ``modified Hermite normal form,'' as detailed in Subsection \ref{subsection-restricted-Hermite-normal-form}.  Again the ``aberrant'' situations can be handled quickly.  We then determine the chromatic number in the remaining cases.  To do so, we subtract rows to produce a homomorphism to a graph with an associated $2\times 2$ matrix, for which we already have a complete theorem.  When this fails to produce a $3$-coloring, we modify the mapping.  We show that whenever we are unable in this manner to properly $3$-color $X$, it must be that in fact $X$ is not properly $3$-colorable, and our procedure instead yields a  $4$-coloring of $X$.  These unusual cases, we show, fall into one of six families, and we state precise numerical conditions for when they occur.  Subsection \ref{subsection-m-3} contains the precise statements and proofs.

As a by-product of our proofs, we show that for the class of graphs we consider, if they don't have loops, then $K_5$ (the complete graph on $5$ vertices) is the only obstacle to $4$-colorability, and ``diamond lanyards'' and ``$C_{13}(1,5)$'' (both of which we define in Subsection \ref{subsection-diamond-lanyards}) are the only obstacles to $3$-colorability.

In Subsection \ref{subsection-algorithm}, we furnish a synopsis of the procedures involved --- an algorithm to determine the chromatic number whenever the Heuberger matrix is of size $m\times 1, 1\times r, 2\times r$, or $3\times 2$.  At \cite{Cervantes-Krebs} we provide a \emph{Mathematica} notebook in which this algorithm is implemented.

% In Section \ref{section-preliminaries}, we saw that an integer distance graph $\text{Cay}(\mathbb{Z},\{\pm a, \pm b, \pm c\})$ has an associated $3\times 2$ matrix.  From our ``$m=3$ theorem,'' then, we can determine the chromatic number of all such graphs.  Thus we have, as we discuss in Section \ref{section-Zhus-theorem}, an alternate proof of a theorem due to Zhu.

\subsection{The case \tops{$m=1$}{m=1}}\label{subsection-m-is-1}

The case $m=1$ can be dealt with immediately.

\begin{Lem}\label{lemma-m-is-1}
Suppose $X$ is a standardized abelian Cayley graph defined by $(y_1\;\cdots\;y_r)^{\text{SACG}}_X$ for some integers $y_1\,\dots,y_r$, not all $0$.  Let $e=\gcd(y_1\,\dots,y_r)$.  Then $X$ has loops (and therefore is not properly colorable) if and only if $e=1$; otherwise, we have that $\chi(X)=2$ if $e$ is even, and $\chi(X)=3$ if $e$ is odd.  
\end{Lem}\begin{proof}  Applying column operations as in \cite[Lemma 2.6]{Cervantes-Krebs-general-publilshed}, we can in effect perform the Euclidean algorithm so as to acquire an isomorphic standardized abelian Cayley graph $(e\;0\;\cdots\;0)^{\text{SACG}}_{X'}=(e)^{\text{SACG}}_{X'}$.  The result follows from \cite[Example 2.1]{Cervantes-Krebs-general-publilshed}.
\end{proof}

If $y_1=\cdots=y_r=0$, then $\chi(X)=2$, by \cite[Lemma 2.11]{Cervantes-Krebs-general-publilshed}.

\subsection{Diamond lanyards and $C_{13}(1,5)$}\label{subsection-diamond-lanyards}

Recall that a \emph{diamond} is a graph with $4$ vertices, exactly one pair of which consists of nonadjacent vertices.  In other words, a diamond is a $K_4$ (complete graph on $4$ vertices) with one edge deleted.  The two vertices not adjacent to one another are the \emph{endpoints} of the diamond.

\begin{Def}\label{def-diamond-lanyard}
An \emph{unclasped diamond lanyard of length} $1$ is a diamond.  Recursively, we define an \emph{unclasped diamond lanyard $U$ of length $\ell+1$} to be the union of an unclasped diamond lanyard $Y$ of length $\ell$ and a diamond $D$, such that $Y$ and $D$ have exactly one endpoint in common.  The \emph{endpoints} of $U$ are the endpoint of $Y$ which is not an endpoint of $D$, and the endpoint of $D$ which is not an endpoint of $Y$.  A \emph{(clasped) diamond lanyard of length} $\ell$ is obtained by adding to an unclasped diamond lanyard $U$ of length $\ell$ an edge between the endpoints of $U$.  We call that edge a \emph{clasp}.
\end{Def}

Def. \ref{def-diamond-lanyard} does not preclude the possibiity of the diamonds in the lanyard having edges in common.  For example, let $X=\text{Cay}(\mathbb{Z}_5,\{\pm 1, \pm 2\})$.  Then $X$ is a complete graph with vertex set $\{0, 1, 2, 3, 4\}$.  We write an edge in $X$ as a string of two vertices.  So $X$ contains as a subgraph a diamond with edges $01, 02, 12, 13, 23$ and endpoints $0$ and $3$.  It also contains as a subgraph a diamond with edges $41, 42, 12, 13, 23$ and endpoints $4$ and $3$.  The edge sets of these two diamonds are non-disjoint.  Taking the union of these two diamonds produces an unclasped diamond lanyard of length two with endpoints $0$ and $4$.  Observing that $04$ is also an edge in $X$, indeed $X$ contains a clasped diamond lanyard as a subgraph.

We sometimes refer to a clapsed diamond lanyard simply as a \emph{diamond lanyard}.  A diamond lanyard  of length $1$ is a $K_4$, and a diamond lanyard  of length $2$ where the two diamonds have disjoint edges is called a \emph{Mosers' spindle} \cite{Soifer}.  Figure \ref{diamond-lanyard} illustrates a diamond lanyard  of length $4$.

\begin{figure}[htp]
\centering
    \includegraphics[scale=.3]{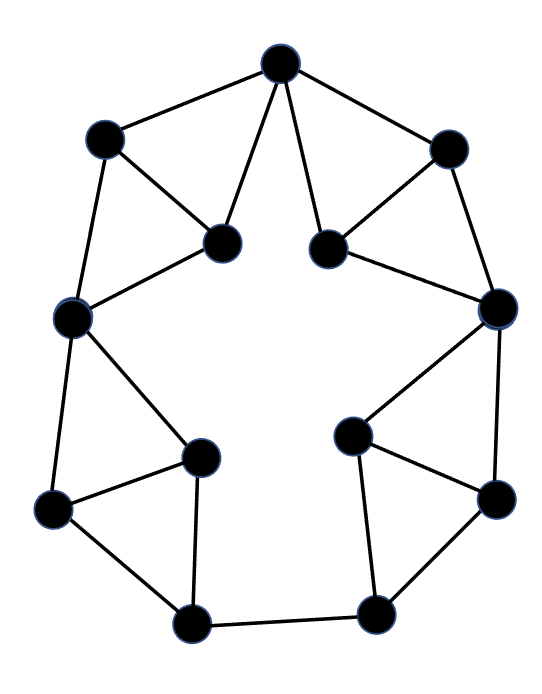}
    \caption{A diamond lanyard of length $4$}
    \label{diamond-lanyard}
\end{figure}

(We would have preferred to have dubbed these ``diamond chain necklaces,'' but the term ``necklace graph'' already has a standard meaning.)

We observe that diamond lanyards are not $3$-colorable.  For suppose we have a proper $3$-coloring.  Note that in any proper $3$-coloring of a diamond, its endpoints must have the same color.  Hence the endpoints of the diamond lanyard must have the same color, but they are adjacent, which is a contradiction.  Thus, we have the following lemma.

\begin{Lem}\label{lemma-chi-of-diamond-lanyard}
Suppose $X$ is a graph containing a diamond lanyard as a subgraph.  Then $\chi(X)\geq 4$.
\end{Lem}

We remark that a diamond lanyard of length $\geq 2$, where the defining diamonds have mutually disjoint edge sets, is a unit-distance graph; hence the chromatic number of the plane is at least $4$.  (Indeed, as discussed in the introduction of \cite{Cervantes-Krebs-general-publilshed}, it is now known to be at least $5$.)

One more remark: In his master's thesis \cite{Tim}, Harris introduces the \emph{Diamond is Forever subgroup} of an abelian group $G$ induced by a symmetric subset $S$ of $G$.  This is the subgroup generated by the endpoints of diamonds in the corresponding Cayley graph.  Our proof of Lemma \ref{lemma-chi-of-diamond-lanyard} shows that if some element of the Diamond is Forever subgroup is in $S$, then $\chi(\text{Cay}(G,S))\geq 4$.  Our later uses of Lemma \ref{lemma-chi-of-diamond-lanyard} in this paper can be reformulated in this manner, if desired.

\vspace{.1in}

We now turn our attention to the other object that can stand in the way of $3$-colorability for our graphs.  Let $C_{13}(1,5)$ be the circulant graph $\text{Cay}(\mathbb{Z}_{13},\{\pm 1, \pm 5\})$.  (For more about this notation, see Def. \ref{def-Heuberger-circulant}.)  Heuberger proves the following lemma in \cite{Heuberger} using a ``vertex-chasing'' argument.  Here we prove it by showing that the independence number is $4$.  We include this proof so that it might suggest generalizations.

\begin{Lem}\label{lemma-chi-of-C-13-1-5}
The chromatic number of $C_{13}(1,5)$ is $4$.
\end{Lem}\begin{proof}It is straightforward to find a proper $4$-coloring of $C_{13}(1,5)$.   We now let $A$ be an independent set of vertices in $C_{13}(1,5)$.  We will show that $|A|\leq 4$.  This will prove the lemma.

Suppose that $|A|\geq 5$; we will show that this leads to a contradiction.  First observe that we must have $|A|\leq 6$.  For if $|A|\geq 7$, then $A$ would contain the adjacent vertices $x$ and $x+1$ for some $x\in\mathbb{Z}_{13}$.  Because $|A|\leq 6<\frac{13}{2}$, there must be $x\in\mathbb{Z}_{13}$ such that $x\notin A$ and $x+1\notin A$.  From $|A|\geq 5$ we find that \[\{x+2,x+3,x+4,x+5,x+6,x+7\}\cap A \text{ or } \{x+8,x+9,x+10,x+11,x+12\}\cap A\] must contain at least $3$ elements.  Because $z$ and $z+1$ are adjacent for all $z\in\mathbb{Z}_{13}$, these three elements must be $x+y-2, x+y,$ and $x+y+2$ for some $y\in\{4, 5, 10\}$.  Using the fact that $a\pm 1\notin A$ and $a\pm 5\notin A$ whenever $a\in A$, we see that \[x+y+1, x+y+3, x+y+5, x+y+6, x+y+7, x+y+8, x+y+10, x+y+12\notin A.\]  Because $|A|\geq 5$, we know that $A$ must contain at least two elements other than  $x+y-2, x+y,$ and $x+y+2$, but the only remaining elements in $\mathbb{Z}_{13}$ are $x+y+4$ and $x+y+9$.  However, $x+y+4$ and $x+y+9$ are adjacent and so cannot both be in $A$.\end{proof}

\subsection{The case \tops{$m=2$}{m=2}}\label{subsection-m-2}

In this subsection, we completely determine $\chi(X)$ when $X$ has a $2\times 2$ matrix as an associated Heuberger matrix.  (Note that if the number of columns exceeds the number of rows, we can perform column operations as in \cite[Lemma 2.6]{Cervantes-Krebs-general-publilshed} to get a zero column, and then delete it.  So the results of this section, together with the Tomato Cage Theorem, completely take care of all dimension $2$ cases.)

Suppose $X$ is a standardized abelian Cayley graph with a $2\times 2$ Heuberger matrix.  By performing row and column operations as in \cite[Lemma 2.6]{Cervantes-Krebs-general-publilshed}, one can show that $X$ is isomorphic to a standardized abelian Cayley graph $X'$ with an associated matrix $M_{X'}$ that is lower-triangular, and such that the diagonal entries of $M_{X'}$ are non-negative.  Hence we may restrict our attention to the case where $M_X$ is a $2\times 2$ lower-triangular matrix with nonnegative diagonal entries.

\vspace{.1in}

Next we delve into a class of graphs that play a crucial role in the main theorem for this subsection.

\begin{Def}\label{def-Heuberger-circulant}Let $a, b$, and $n$ be integers with $n\neq 0$ and $\gcd(a,b,n)=1$ and $n\nmid a$ and $n\nmid b$.  We say that $\text{Cay}(\mathbb{Z}_n,\{\pm a,\pm b\})$ is a \emph{Heuberger circulant}, denoted $C_n(a,b)$.\hfill $\square$\end{Def}

The condition $\gcd(a,b,n)=1$ is equivalent to the connectedness of $C_n(a,b)$, and the condition that $n\nmid a$ and $n\nmid b$ is equivalent to the absence of loops.

In \cite{Heuberger}, Heuberger completely determines the chromatic number of all Heuberger circulants, as follows.

\begin{Thm}[{\cite[Theorem 3]{Heuberger}}]\label{theorem-Heubergers}Let $C_n(a,b)$ be a Heuberger circulant.  Then \[\chi(C_n(a,b)) = \begin{cases}

2 & \text{if }a\text{ and }b\text{ are both odd, but }n\text{ is even}
\\

5 & \text{if }n=\pm 5\text{ and }a\equiv \pm 2b\Mod 5\\

4 & \text{if }n=\pm 13,\text{ and }\; a\equiv \pm 5b\Mod {13}\\

4 & \text{if }(i)\; n\neq\pm 5, \text{ and } (ii)\;3\nmid n,\text{ and }(iii)\; a\equiv \pm 2b\Mod {n}\text{ or  }b\equiv \pm 2a\Mod {n}\\

3 & \text{otherwise.}\end{cases}\]\end{Thm}

We note that \cite{Heuberger} excludes the case $a\equiv \pm b\Mod n$, in which event $C_n(a,b)$ is an $n$-cycle.  However, the theorem as stated here includes this possibility as well.

Observe that the third case in Theorem \ref{theorem-Heubergers} is precisely Lemma \ref{lemma-chi-of-C-13-1-5}.  Moreover, for the fourth case, Lemma \ref{lemma-chi-of-diamond-lanyard} shows that $\chi(C_n(a,b))\geq 4$, for the following reason.  First suppose that $a\equiv 2b\Mod{n}$.  Observe that for all $x\in\mathbb{Z}_n$, there is a diamond in $C_n(a,b)$ with vertex set $\{x, x+a, x+2a, x+3a\}$, where the endpoints are $x$ and $x+3a$.  Because $a\equiv 2b\Mod{n}$ and $\gcd(a,b,n)=1$, we must have that $\gcd(a,n)=1$.  Moreover, because $3\nmid n$, we have that $\gcd(3a,n)=1$.  Let $m$ be a positive integer such that $3am\equiv 1\Mod{n}$.  Then $C_n(a,b)$ contains a diamond lanyard with vertex set $\{ja\;|\;0\leq j\leq m\}$, so by Lemma \ref{lemma-chi-of-diamond-lanyard}, we have that $\chi(C_n(a,b))\geq 4$.  A similar argument works when $a\equiv -2b\Mod{n}$ or when $b\equiv \pm 2a\Mod{n}$.  This is essentially the reasoning in \cite{Heuberger} for the lower bounds in these cases.  Hence a Heuberger circulant is $4$-colorable unless it is $K_5$; and it is $3$-colorable unless it contains as a subgraph either a diamond lanyard or it equals $C_{13}(1,5)$.  (Recall that $K_5$ contains a diamond lanyard, so we needn't include it separately in the list of obstructions to $3$-colorability.)

% We note that Theorem 3.4 has the following structure.  Putting aside bipartite graphs, which can be handled easily, we have that the chromatic number is typically $3$, with some explicitly described uncommon exceptions.  Zhu's theorem on distance graphs (THEOREM WHATEVER BELOW) has the same kind of format, as do many other theorems about these sorts of graphs.  These all seem to jibe with Alon's result in \cite{Alon} which, roughly speaking, says that for certain circulants, a Cayley graph on a randomly chosen set of $k$ elements has chromatic number $3$ asympotically almost surely.

Next, we show that when the entries of the first column of $M_X$ are relatively prime, then with some mild additional conditions imposed, $X$ is isomorphic to a Heuberger circulant.

\begin{Lem}\label{lemma-sufficient-condition-to-be-properly-given-circulant}Let $X$ be a standardized abelian Cayley graph with associated Heuberger matrix \[M_X=\begin{pmatrix}
y_{11} & y_{12}\\
y_{21} & y_{22}
\end{pmatrix}.\]Suppose that $X$ does not have loops, that $\det(M_X)\neq 0$, and that $\gcd(y_{11},y_{21})=1$.  Then $X$ is isomorphic to the Heuberger circulant $C_n(a,b)$, where $a=-y_{21}$, $b=y_{11}$, and $n=\det(M_X)$. \end{Lem}
\begin{proof}Define a homomorphism $\phi\colon \mathbb{Z}^2\to \mathbb{Z}_n$ by $e_1\mapsto a, e_2\mapsto b$.  Let $y_1$ and $y_2$ be the first and second columns, respectively, of $M_X$.  We must show that $\ker\phi$ equals the $\mathbb{Z}$-span of $y_1$ and $y_2$.  It is straightforward to show that $\phi(y_1)=\phi(y_2)=0$, giving us one inclusion.  We now show the reverse inclusion.  Suppose $\phi((x_1,x_2)^t)=0$.  Then $ax_1+bx_2\equiv 0\Mod n$.  Because $a$ and $b$ are relatively prime, there exist integers $q,r$ such that $aq+br=1$.  Using elementary number theory as in \cite{Pommersheim}, we find that $x_1=\ell nq+kb$ and $x_2=\ell nr-ka$ for some integers $\ell, k$.  Hence after some computations we find that\[\begin{pmatrix}
x_1\\ x_2
\end{pmatrix}=ky_1+\ell n\begin{pmatrix}
q\\ r
\end{pmatrix}=ky_1+\begin{pmatrix}
\ell qy_{11}y_{22}-\ell qy_{21}y_{12}\\
\ell r y_{11}y_{22}-\ell r y_{21}y_{12}
\end{pmatrix}=(k+\ell qy_{22}-\ell ry_{12})y_1+\ell y_2.\qedhere\]\end{proof}

Note that after switching the columns of the matrix, Lemma \ref{lemma-sufficient-condition-to-be-properly-given-circulant} becomes a special case of \cite[Example 2.5]{Cervantes-Krebs-general-publilshed}.  We have included it here separately so as to give a more elementary proof for $2\times 2$ matrices.

We remark that not every graph with an associated $2\times 2$ Heuberger matrix is isomorphic to a circulant.  The graph \(\begin{pmatrix}
4 & 0\\
2 & 4
\end{pmatrix}_X^{\text{SACG}}\), for instance, provides a counterexample.  One computes that $X$ must have order $16$ and so, if circulant, would be of the form $\text{Cay}(\mathbb{Z}_{16},S)$ for some $S$.  Because $X$ is connected and bipartite, and has degree $4$, we see that up to isomorphism, we can assume $S=\{\pm 1, \pm 3\}$ or $S=\{\pm 1,\pm 7\}$.  Direct arguments (for example, counting the number of paths of length $2$ between various vertices) show that neither choice of $S$ produces a circulant graph isomorphic to $X$.

\vspace{.05in}

The following lemma characterizes $2\times 2$ matrices whose associated graphs have loops.  This occurs either when the determinant is nonzero and divides a row, or else when one row is zero and the other has relatively prime entries.

\begin{comment}
\begin{Lem}\label{lemma-sufficient-condition-to-be-properly-given-circulant}Let $X$ be a standardized abelian Cayley graph defined by \[\begin{pmatrix}
y_{11} & 0\\
y_{21} & y_{22}
\end{pmatrix}_X.\]Suppose that $y_{22}\nmid y_{21}$, that $y_{11}>0$ and $y_{22}>0$, and that $\gcd(y_{11},y_{21})=1$.  Then $X$ is isomorphic to the Heuberger circulant $C_n(a,b)$, where $a=-y_{21}$, $b=y_{11}$, and $n=y_{11}y_{22}$.\end{Lem}\begin{proof}Because $y_{22}\nmid y_{21}$, we have that $y_{22}\neq 1$.  So $n\nmid a$ and $n\nmid b$.  Also, because $\gcd(y_{11},y_{21})=1$, we have that $\gcd(a,b,n)=1$.

The result now follows by computing the kernel of the mapping $\phi\colon\mathbb{Z}^2\to \mathbb{Z}_n$ defined by $\phi\colon e_1\mapsto a, e_2\mapsto b$.  The column vectors $y_1=(y_{11}, y_{21})^t$ and $y_2=(0, y_{22})^t$ are certainly in this kernel.  Conversely, we must show that if \begin{equation}\label{equation-kernel-phi}
    c_1a+c_2b\equiv 0\Mod{y_{11}y_{22}},
\end{equation} then $(c_1,c_2)^t$ is in the $\mathbb{Z}$-span of $\{y_1,y_2\}$.  From (\ref{equation-kernel-phi}) and the relative primality of $y_{11}$ and $y_{21}$ we get that $ky_{11}=c_1$ for some integer $k$.  Substituting this back into (\ref{equation-kernel-phi}), we find that $$ky_{11}(-y_{21})+c_2y_{11}\equiv 0\Mod{y_{11}y_{22}},$$ so $\ell y_{22}=c_2-ky_{21}$ for some integer $\ell$.  Hence $(c_1,c_2)^t=ky_1+\ell y_2$.
\end{proof}\end{comment}

\begin{Lem}\label{lemma-loops-2-by-2}Let $X$ be a standardized abelian Cayley graph with an associated matrix \[M_X=\begin{pmatrix}
    y_{11} & y_{12}\\
    y_{21} & y_{22}
\end{pmatrix} \]  Let $n=\det(M_X)$.  If $n\neq 0$, then $X$ has loops if and only if either (i) $n\mid y_{11}$ and $n\mid y_{12}$ or (ii) $n\mid y_{21}$ and $n\mid y_{22}$.  If $n=0$, then $X$ has loops if and only if either (a) $y_{11}=y_{12}=0$ and $\gcd(y_{21},y_{22})=1$, or (b)~$y_{21}=y_{22}=0$ and $\gcd(y_{11},y_{12})=1$.\end{Lem}

\begin{proof} First suppose $n\neq 0$.  We know that $X$ has loops if and only if $e_1$ or $e_2$ is in the $\mathbb{Z}$-span $H$ of the columns of $M_X$.  We have $e_1\in H$ if and only if $M_X^{-1}e_1\in\mathbb{Z}^2$.  But \[M_X^{-1}e_1=\frac{1}{n}\begin{pmatrix}
    y_{22} & -y_{21}\\
    -y_{12} & y_{11}
\end{pmatrix}e_1=(y_{22}/n\;\;\;\;-y_{21}/n)^t.\]  Similarly for $e_2$.

Now suppose $n=0$.  If (a) holds, then $e_2\in H$.  If (b) holds, then $e_1\in H$.   Conversely, suppose that $X$ has loops, so that $e_1\in H$ or $e_2\in H$.  First suppose $e_1\in H$.  Then there exist $r,s\in\mathbb{Z}$ such that \begin{equation}\label{equation-e1-in-H}r\begin{pmatrix}
    y_{11}\\
    y_{21}
\end{pmatrix}+s\begin{pmatrix}
    y_{12}\\
    y_{22}
\end{pmatrix}=e_1,\end{equation} so $ry_{21}+sy_{22}=0$.  Because $n=0$, the left column and right columns of $M_X$ are linearly dependent over $\mathbb{Q}$.  Then there exist $\alpha, \beta\in\mathbb{Z}$, not both $0$, such that \begin{equation}\label{equation-cols-lin-dep}\alpha\begin{pmatrix}
    y_{11}\\
    y_{21}
\end{pmatrix}=\beta\begin{pmatrix}
    y_{12}\\
    y_{22}
\end{pmatrix}.\end{equation}  We can assume $\gcd(\alpha,\beta)=1$, for if not, replace them with $\alpha/\gcd(\alpha,\beta)$ and $\beta/\gcd(\alpha,\beta)$, respectively.

Then $r\beta y_{21}+s\beta y_{22}=0$, so $r\beta y_{21}+s\alpha y_{21}=0$, so $r\beta +s\alpha=0$ or $y_{21}=0$.

Case 1: Suppose $r\beta +s\alpha=0$.  From (\ref{equation-e1-in-H}) we also have $ry_{11}+sy_{12}=1$, so $\gcd(r,s)=1$.  We have $r\beta=-s\alpha$, so $\alpha\mid r$ and $\beta\mid s$.  Also $r\mid \alpha$ and $s\mid \beta$.  So either $\alpha=r$ and $\beta=-s$, or $\alpha=-r$ and $\beta=s$.  Either way, (\ref{equation-e1-in-H}) and (\ref{equation-cols-lin-dep}) now contradict one another.

Case 2: Suppose $y_{21}=0$.  Then $0=\alpha y_{21}=\beta y_{22}$.  If $y_{22}=0$, then using (\ref{equation-e1-in-H}), we find that (b) holds.  If not, then $\beta=0$, so $\alpha\neq 0$, so $y_{11}=y_{21}=0$.  Equation (\ref{equation-e1-in-H}) then implies that (b) holds.

A similar argument shows that if $e_2\in H$, then (a) must hold.\end{proof}

We now have all the pieces in place needed to find $\chi(X)$ whenever $X$ has an associated $2\times 2$ matrix.

\begin{Thm}\label{theorem-m-equals-2}Let $X$ be a standardized abelian Cayley graph defined by \[\begin{pmatrix}
y_{11} & 0\\
y_{21} & y_{22}
\end{pmatrix}^{\emph{SACG}}_X.\]Suppose that $y_{11}\geq 0$ and $y_{22}\geq 0$. Let $d=\gcd(y_{11},y_{21})$ and $e=\gcd(y_{11},y_{21},y_{22})$.  Then:\begin{enumerate}

    \item\label{statement-m-2-loops} If either (i) $y_{22}=1$ or (ii) $y_{11}=1$ and $y_{22}|y_{21}$ or (iii) $y_{11}=0$ and $\gcd(y_{21},y_{22})=1$, then $X$ has loops and is not properly colorable.
    
    \item\label{statement-m-2-bipartite} If both $y_{11}+y_{21}$ and $y_{22}$ are even, then $\chi(X)=2$.
    
    \item\label{statement-m-2-top-row-or-second-column-zero-or-e-greater-than-1-or-y22-divides-y11} If (i) neither of the conditions in the previous statements holds, and (ii) $y_{11}=0$ or $y_{22}=0$ or $e>1$ or $y_{22}|y_{21}$, then $\chi(X)=3$.

    \item\label{statement-m-2-Heuberger} If none of the conditions in the previous statements hold, take $q\in\mathbb{Z}$ such that $\text{gcd}(y_{11},y_{21}+qy_{22})=1$.  (Such a $q$ necessarily exists, as we can let $q$ be the product of all primes $p$ such that $p|y_{11}$ but $p\nmid d$.  Here we adopt the convention that if there are no such primes, let $q=1$.)  Then $\chi(X)=\chi(C_n(a,b))$, where $a=-y_{21}-qy_{22}$, $b=y_{11}$, and $n=y_{11}y_{22}$.
\end{enumerate}\end{Thm}\begin{proof} Lemma \ref{lemma-loops-2-by-2} implies both Statement (\ref{statement-m-2-loops}) and its converse.  Statement (\ref{statement-m-2-bipartite}) follows from \cite[Lemma 2.11]{Cervantes-Krebs-general-publilshed}.

We now prove Statement (\ref{statement-m-2-top-row-or-second-column-zero-or-e-greater-than-1-or-y22-divides-y11}).  Assume that conditions (i) and (ii) in that statement both hold.  Condition (i) implies that $X$ does not have loops, and that $X$ is not bipartite.  If $y_{11}=0$, then we can delete the top row as per \cite[Lemma 2.8]{Cervantes-Krebs-general-publilshed} without affecting the chromatic number.  Lemma \ref{lemma-m-is-1} now gives us $\chi(X)=3$.  If $y_{22}=0$, then we can delete the second column as per \cite[Lemma 2.6(4)]{Cervantes-Krebs-general-publilshed} without changing $X$.  We then find that $\chi(X)=3$, by \cite[Theorem 2.15]{Cervantes-Krebs-general-publilshed}.  If $e>1$, then $\chi(X)=3$, by \cite[Lemma 2.12]{Cervantes-Krebs-general-publilshed}.  If $y_{22}|y_{21}$, then after performing a column sum as in \cite[Lemma 2.6(4)]{Cervantes-Krebs-general-publilshed} to eliminate $y_{21}$, by \cite[Lemma 2.7]{Cervantes-Krebs-general-publilshed}, we have that $\chi(X)=3$.

Finally, we prove Statement (\ref{statement-m-2-Heuberger}).  By \cite[Lemma 2.6(3)]{Cervantes-Krebs-general-publilshed}, we have that \[\begin{pmatrix}
y_{11} & 0\\
y_{21} & y_{22}
\end{pmatrix}^{\text{SACG}}_X=\begin{pmatrix}
y_{11} & 0\\
y_{21}+qy_{22} & y_{22}
\end{pmatrix}^{\text{SACG}}_X.\]The result follows from Lemma \ref{lemma-sufficient-condition-to-be-properly-given-circulant}.\end{proof}

% \vspace{.1in}In Theorem \ref{theorem-m-equals-2}(\ref{statement-m-2-Heuberger}), the only significance of $q$ is to guarantee that $\gcd(a,b)=1$.  Any other $q$ with this property will yield the same result.

\vspace{.1in}It follows from Thm. \ref{theorem-m-equals-2} and our previous observations about Heuberger circulants that for a standardized abelian Cayley graph $X$ with an associated $2\times 2$ Heuberger matrix, if $X$ does not have loops, then $X$ is $4$-colorable unless it contains $K_5$ as a subgraph; and it is $3$-colorable unless it contains as a subgraph either a diamond lanyard or $C_{13}(1,5)$.  We will see in \S\ref{subsection-m-3} that this statement holds for the $3\times 2$ case as well.

\vspace{.1in}We remark that an alternate approach to a $2\times 2$ integer matrix is possible.  We now briefly sketch this method.  First we divvy into cases according to whether (i) the matrix is singular, or (ii) the matrix is nonsingular and has a primitive row, i.e., a row whose entries are relatively prime, or (iii) the matrix is nonsingular and has no primitive rows.  In all three cases, we perform column operations as per the steps in the Euclidean algorithm applied to the entries of the top row.  In Case (i) we thereby obtain a zero column, which can be deleted, whereupon we apply the Tomato Cage theorem.  In Case (ii) we obtain a matrix whose top row is $(1\;0)$; the corresponding graph is therefore circulant, so Theorem \ref{theorem-Heubergers} gives us its chromatic number.  In Case (iii) we obtain a homomorphism by appending the columns $d_1e_1$ and $d_2e_2$, where $d_1$ and $d_2$ are the gcds of the entries in the first and second rows, respectively.  We may then delete the original columns, as they lie in the $\mathbb{Z}$-span of $d_1e_1$ and $d_2e_2$.  This is a diagonal matrix, so the chromatic number is the maximum of the chromatic numbers of the graphs associated to $(d_1)$ and $(d_2)$.  As those are both cycle graphs, it follows that the original graph is $3$-colorable.

\vspace{.1in}We now discuss a useful corollary that follows from Theorem \ref{theorem-m-equals-2}.  When $M_X$ is a $2\times 2$ integer matrix, we perform row and column operations to create a $2\times 2$ lower triangular matrix $M_{X'}$ for which $X'$ is isomorphic to $X$.  Observing the effect of these operations on the determinant, we find that $|\det\;M_X|=|\det\;M_{X'}|$.  Note that in Theorem \ref{theorem-m-equals-2}, whenever $X$ has no loops and $\chi(X)>3$, we have that $|\det\;M_X|$ is not divisible by $3$.  Thus we have the following corollary, which provides an easily checked sufficient condition for $3$-colorability.

\begin{Cor}\label{corollary-det-divis-by-3}Let $X$ be a standardized abelian Cayley graph with an associated $2\times 2$ matrix $M_X$.  If $X$ has no loops and $3\mid \det\;M_X$, then $\chi(X)\leq 3$.\end{Cor}

We note that Cor. \ref{corollary-det-divis-by-3} fails in general for larger matrices.  For example, we have by Thm. \ref{theorem-m-equals-2}, Thm. \ref{theorem-Heubergers}, \cite[Lemma 2.7]{Cervantes-Krebs-general-publilshed}, and Example \cite[2.1]{Cervantes-Krebs-general-publilshed} that \[\chi\left(\begin{pmatrix}
1 & 0 & 0\\
-2 & 5 & 0\\
0 & 0 & 3
\end{pmatrix}^{\text{SACG}}_X\right)=5.\]

However, observe that not every $2\times 2$ minor of the matrix above has determinant divisible by $3$.  As mentioned in the introduction, we conjecture that if this stronger condition holds, and $X$ does not have loops, then $X$ is $3$-colorable.  That stronger condition is equivalent to the statement that the rank over $\mathbb{Z}_3$ of the reduction of the matrix modulo $3$ (i.e., the matrix obtained by reducing each of its entries modulo $3$) is $\leq 1$.

% Statements (\ref{statement-m-2-first-column-e-1}) and (\ref{statement-m-2-second-column-e-2}) follow from Lemma \ref{lemma-loops}.  Statement (\ref{statement-m-2-Heuberger}) follows from Lemma \ref{lemma-sufficient-condition-to-be-properly-given-circulant} and Theorem \ref{theorem-Heubergers}.  For the remainder of this proof, we assume that none of the conditions in statements (1)--(4) hold.  In particular, by Lemma \ref{lemma-bipartite}, we have that $\chi(X)\geq 3$.

% If $y_{22}=0$, then $y_{11}+y_{21}$ is odd, because the condition in statement (\ref{statement-m-2-bipartite}) does not hold.  Moreover, $X$ does not have loops, because the conditions in statements (\ref{statement-m-2-first-column-e-1}) and (\ref{statement-m-2-second-column-e-2}) do not hold.  Therefore, $\chi(X)=3$ by the Tree Guard Theorem (Theorem \ref{theorem-tree-guard}).  For the remainder of this proof, we assume that $y_{22}\neq 0$.

% Therefore $e\geq 3$ and $e$ is odd.  We have a graph homomorphism from $X$ to an $e$-cycle $Y$ given by \[\begin{pmatrix}y_{11} & 0\\ y_{21} & y_{22}\end{pmatrix}_X \xrightarrow{\ocirc} \begin{pmatrix}y_{11}+y_{21} & y_{22}\end{pmatrix}\xrightarrow{\ocirc} \begin{pmatrix} e & e\end{pmatrix}_Y=(e)_Y.\]. Because $Y$ is $3$-colorable, therefore so is $X$.

\subsection{Modified Hermite normal form}\label{subsection-restricted-Hermite-normal-form}

In Subsection \ref{subsection-m-2}, we saw that it was useful to deal only with $2\times 2$ Heuberger matrices in a certain convenient format.  The same goes for $3\times 2$ Heuberger matrices.  The crux of this idea is drawn from \cite{Heuberger}, where Hermite normal form is used.  For $3\times 2$ matrices, we refine the requirements slightly for our purposes, so as to further reduce the number of exceptional cases.  The purpose of the present subsection is to define this ``modified Hermite normal form'' for $3\times 2$ matrices and to show that with very few exceptions every standardized abelian Cayley graph with an associated $3\times 2$ Heuberger matrix is isomorphic to one with a matrix in this form.  We do not attempt here to generalize these definitions to matrices of arbitrary size, as we do not know yet what restrictions will prove to be most useful when the rank or dimension is larger.

\begin{Def}\label{def-modified-Hermite-normal-form}Let  $$M=\begin{pmatrix}
y_{11} & y_{12}\\
y_{21} & y_{22}\\
y_{31} & y_{32}
\end{pmatrix}$$ be a $3\times 2$ matrix with integer entries such that no row of $M$ has all zero entries.  We say $M$ is in \emph{modified Hermite normal form} if the following conditions hold:\begin{enumerate}
    \item $y_{11}>0$, and %, by multiplying row $1$ by $-1$ if need be.

\item $y_{12}=0$, and %, by applying column operations a la Euclidean algorithm to top row.

\item $y_{11}y_{22}\equiv y_{11}y_{32}\Mod{3}$, and

\item  $y_{22}\leq y_{32}$, and %(by switching rows 2 and 3 if need be, doesn't affect whether $3\mid M_Y$).  So in particular $y_{32}\geq 0$.

\item  $|y_{22}|\leq |y_{32}|$, and 

\item Either (i) $y_{22}=0$ and $-\frac12 |y_{32}|\leq y_{31}\leq 0$, or else (ii) $-\frac12 |y_{22}|\leq y_{21}\leq 0$. %, by using column operations a la restricted Hermite normal form.

\end{enumerate}\end{Def}

% MAKE IT MODIFIED NOT RESTRICTED.  JUST FOR $m=3$.  HAVE $y_{22}\leq y_{32}$ AND $|y_{22}|\leq |y_{32}|$.

There are some departures here from the usual Hermite normal form.  For example, $|y_{21}|$ cannot be more than half of $y_{22}$, a more stringent requirement than being less than $y_{22}$, as in ordinary Hermite normal form.  As we shall see, we can impose this narrower condition because we have both row and column operations at our disposal, not just column operations.  Moreover, this form may be the transpose of what some readers are accustomed to, but we have adopted this convention so as to be consistent with \cite{Heuberger}.  The third condition has no analogue with Hermite normal form but will turn out to be rather useful in the succeeding subsection.

We next show that every standardized abelian Cayley graph with a $3\times 2$ Heuberger matrix of rank 2 without zero rows is isomorphic to one with a Heuberger matrix in modified Hermite normal form.

\begin{Lem}\label{lemma-modified-Hermite-normal-form}Let $X$ be a standardized abelian Cayley graph with a $3\times 2$ Heuberger matrix $M_X$.  Suppose that $M_X$ has no zero rows, and that the columns of $M_X$ are linearly independent over $\mathbb{Q}$.  Then $X$ is isomorphic to a standardized abelian Cayley graph $X'$ with a $3\times 2$ Heuberger matrix $M_{X'}$ in modified Hermite normal form.\end{Lem}\begin{proof}The proof is constructive.  We give an explicit algorithm for row and column operations to perform on $M_X$, as per \cite[Lemma 2.6]{Cervantes-Krebs-general-publilshed}, so as to result in the desired matrix $M_{X'}$.

% If $M_X$ has a zero column (that is, a column in which every entry is zero), then by Lemma \ref{lemma-isomorphisms}(\ref{item-delete-column-Z-span}), we may delete the zero column without changing $X$, giving us an $m\times 1$ matrix.  Hence we may assume that $M_X$ has no zero columns.

\begin{enumerate}
    % \item[Step One:] If $M_X$ has any zero rows (that is, rows in which all entries are zero), permute the rows so that the zero rows appear at the bottom of the matrix.  Let $M_{X_1}$ be the resulting matrix.  In $M_{X_1}$, then, a zero row will never appear above a row which is not a zero row.  Because $M_X$ has no zero columns, the top row of $M_{X_1}$ is not a zero row.

\item[Step Zero:] Let $\alpha, \beta$, and $\gamma$ be the determinants of the $2\times 2$ minors of $M_X$.  Two of $\alpha, \beta, \gamma$ must be congruent to each other or negatives of each other modulo $3$.  Using this fact, we can permute rows and/or multiply rows by $-1$ as needed so that the determinant of the submatrix formed by the top two rows is congruent modulo $3$ to the submatrix formed by the first and third rows.  This property will be preserved by all subsequent steps and will eventually lead to satisfaction of the third condition in Def. \ref{def-modified-Hermite-normal-form}.  Let $M_{X_{1}}$ be the resulting matrix.

\item[Step One:] If the first entry of any column of $M_{X_{1}}$ is negative, multiply that column (or those columns) by $-1$.  Let $M_{X_2}$ be the resulting matrix.  The top row of $M_{X_2}$ has no negative entries.

\item[Step Two:] If the first entry of the first column of $M_{X_2}$ is $0$, then permute the two columns; otherwise, do nothing.  Let $M_{X_3}$ be the resulting matrix.  The first entry of the first column of $M_{X_3}$ is strictly positive.  (Here is where we use the assumption that $M_X$ has no zero rows.)  If the first entry of the second column is $0$, then let $M_{X_4}=M_{X_3}$ and skip to Step Four.

\item[Step Three:] Both entries of the top row of $M_{X_3}$ are strictly positive.  Let $e$ be the greatest common divisor of these two entries.  Repeatedly applying \cite[Lemma 2.6(3)]{Cervantes-Krebs-general-publilshed}, we perform column operations that effectuate the Euclidean algorithm on the entries in the top row of $M_{X_3}$, so that the top row of the resulting matrix has two entries, one of which is $e$, the other $0$.  If the first entry of the first column is $0$, then permute the two columns.  Let $M_{X_4}$ be the resulting matrix.  The top row of $M_{X_4}$ has a strictly positive first entry, and $0$ for its second entry.  The matrix $M_{X_4}$ now meets the first three conditions in the definition of modified Hermite normal form, and these will be preserved by all subsequent steps.

\item[Step Four:] Let $z$ and $w$ be the $(3,1)$ and $(3,2)$ entries of $M_{X_4}$, respectively.  If $z\leq w$ and $|z|\leq |w|$, then do nothing.  If $z>w$ and $|z|\leq |w|$, then multiply the second column by $-1$.  If $z\leq w$ and $|z|> |w|$, then switch the bottom two rows and multiply the second column by $-1$.  If $z> w$ and $|z|> |w|$, then switch the bottom two rows.  Whatever action was taken, let $M_{X_{5}}$ be the resulting matrix.  The first five conditions in Def. \ref{def-modified-Hermite-normal-form} are now (and will continue to be) met.

% \item[Step Five:] If the second column is a zero column, then delete it, and the result will be the desired $m\times 1$ matrix.  Otherwise, permute rows $2$ through $m$ of $M_{X_5}$ as needed so that the nonzero entries of the second column appear in nondecreasing order, and so that all zero entries of the second column except the top entry appear below these nonzero entries.  Let $M_{X_6}$ be the resulting matrix.  The matrix $M_{X_6}$ now meets the first three conditions in the definition of restricted Hermite normal form.

% \item[Step Five:] Because the columns of $M_X$ are linearly independent over $\mathbb{Q}$, neither column of $M_{X_5}$ is a zero column.  If the second entry of the second column of $M_{X_5}$ is nonzero, then let $M_{X_6}=M_{X_5}$ and move to the next step.  Otherwise, switch the bottom two rows and multiply the second row by so that the second entry of the second row is nonzero, and let $M_{X_5}$ be the resulting matrix.

\item[Step Five:] Let $a$ and $b$ be the $(2,1)$ and $(2,2)$ entries of $M_{X_5}$, respectively.  If $b=0$, then apply to the third rather than second row the procedure described in the rest of Step Five as well as in Step Six.  (Here is where we use that the columns of $M_X$ are linearly independent over $\mathbb{Q}$; this guarantees that if $b=0$, then the $(3,2)$ entry of $M_{X_5}$ is not zero.)  If $b\neq 0$, by the division theorem, there exist integers $q$ and $r$ such that $r=a-q|b|$, where $-|b|<r\leq 0$.  Applying \cite[Lemma 2.6(3)]{Cervantes-Krebs-general-publilshed}, perform a column operations to replace the first column with the first column plus $\pm q$ times the second column, so that the second entry in the first column becomes $r$.  Let $M_{X_6}$ be the resulting matrix.

\item[Step Six:] Let $c$ be the $(2,1)$ entry of $M_{X_6}$.  We still have that $b$ is the $(2,2)$ entry of $M_{X_6}$.  Suppose that $-\frac{|b|}{2}> c$.  We then add $b/|b|$ times the second column to the first column; then multiply the first column by $-1$; and then multiply the first row by $-1$.  Let $M_{X'}$ be the resulting matrix.
\end{enumerate}The matrix $M_{X'}$ will then satisfy all conditions in the definition of modified Hermite normal form.\end{proof}

Suppose we have a $3\times 2$ Heuberger matrix $M_X$.  If $M_X$ has a zero row, then by \cite[Lemma 2.8]{Cervantes-Krebs-general-publilshed}, we can delete it without affecting the chromatic number $\chi$, whereupon Thm. \ref{theorem-m-equals-2} can be used to find $\chi$.  If the columns of $M_X$ are linearly dependent over $\mathbb{Q}$, then as per \cite[Lemma 2.6]{Cervantes-Krebs-general-publilshed} appropriate column operations that do not change $X$ will produce a zero column, which can be deleted without changing $X$.  \cite[Thm. 2.15]{Cervantes-Krebs-general-publilshed} can then be used to find $\chi(X)$.  Otherwise, in light of Lemma \ref{lemma-modified-Hermite-normal-form}, we lose no generality by assuming that $M_X$ is in modified Hermite normal form.

We next show that when $M_X$ is in modified Hermite normal form, we can determine immediately whether $X$ has loops.  Recall that $e_j$ is the $j$th standard basis vector, with a $1$ as its $j$th entry and $0$ for every other entry.

\begin{Lem}\label{lemma-loops}Let $X$ be a standardized abelian Cayley graph with a Heuberger matrix $M_X$.  Suppose that $M_X$ is a $3\times 2$ matrix in modified Hermite normal form.  Then $X$ has loops if and only if either the first column of $M_X$ is $e_1$, or the second column of $M_X$ is $e_3$.\end{Lem}\begin{proof}Let $H$ be the $\mathbb{Z}$-span of the columns of $M_X$.  The graph $X$ has loops if and only if $\pm e_j\in H$ for some~$j$.  From the definition of modified Hermite normal form, we see that this can occur if and only if either the first column of $M_X$ is $e_1$ or the second column of $M_X$ is $e_3$.
\end{proof}

%%%%%%%%%%%%%%%%%%%%%%%%%%%%%%

\subsection{Chromatic numbers of graphs with $3\times 2$ matrices in modified Hermite normal form}\label{subsection-m-3}

In this section we prove the following theorem, which completely determines the chromatic number of an arbitrary standardized abelian Cayley graph with a $3\times 2$ Heuberger matrix in modified Hermite normal form.

\begin{Thm}\label{theorem-m-equals-3}  Let $X$ be a standardized abelian Cayley graph with a Heuberger matrix \[M_X=\begin{pmatrix}
y_{11} & 0\\
y_{21} & y_{22}\\
y_{31} & y_{32}\\
\end{pmatrix}\] in modified Hermite normal form.

\begin{enumerate}

\item If the first column of $M_X$ is $e_1$ or the second column of $M_X$ is $e_3$, then $X$ has loops and cannot be properly colored.

\item If $y_{11}+y_{21}+y_{31}$ and $y_{22}+y_{32}$ are both even, then $\chi(X)=2$.

\item\label{item-six-bad-cases} If \[M_X=\begin{pmatrix}
1 & 0\\
0 & 1\\
\pm 3k & 1+3k
\end{pmatrix}\text{ or }M_X=\begin{pmatrix}
1 & 0\\
0 & -1\\
\pm 3k & -1+3k
\end{pmatrix}\text{ or }M_X=\begin{pmatrix}
1 & 0\\
-1 & 2\\
-1-3k & 2+3k
\end{pmatrix}\text{ or }M_X=\begin{pmatrix}
1 & 0\\
-1 & -2\\
-1+3k & -2+3k
\end{pmatrix}\]

\[\text{ for some  positive integer }k,\text{ or } M_X=\begin{pmatrix}
1 & 0\\
0 & -1\\
3b & 2
\end{pmatrix}\text{ for some  integer }b, \text{ or}\]

\[M_X=\begin{pmatrix}
1 & 0\\
-1 & a\\
-1 & a+3(k-1)
\end{pmatrix}\text{ for some integer }a\text{ with }3\nmid a\text{ and some positive integer }k,\] then $\chi(X)=4$.

\item If none of the above conditions hold, then $\chi(X)=3$.

\end{enumerate}\end{Thm}

The main idea behind the proof of Theorem \ref{theorem-m-equals-3} is to add or subtract two rows as per \cite[Lemma 2.10]{Cervantes-Krebs-general-publilshed} to obtain a homomorphism from $X$ to a graph with a $2\times 2$ Heuberger matrix.  Theorem \ref{theorem-m-equals-2} and its corollary provide conditions under which this latter graph (and hence $X$) is $3$-colorable.  If those conditions are not met, then we get information about $M_X$.  In this case we can repeat the procedure using some other homomorphism to further narrow down the possibilities for those $M_X$ for which $\chi(X)>3$.  In particular, we are led in this way to consider three special types of matrices $M_X$: those with $y_{22}=0$ (we call these ``L-shaped'' matrices); those with $y_{11}=y_{22}=1$ and $y_{21}=0$ (we call these ``$I$ on top'' matrices); and those with $y_{11}=y_{21}=y_{31}=1$ (we call these ``first column all ones'' matrices).  (We remark that \cite{Tim}, which deals with $4\times 2$ Heuberger matrices, introduces the terminology ``tri-triangle'' for a column in which exactly three entries are $\pm 1$ and the other entries are $0$.  Matrices with such columns play a significant role for the $4\times 2$ case.  The term ``tri-triangle'' alludes to the fact that this column corresponds to a triangle in the graph brought about by three distinct generators.) Every exceptional case traces back ultimately to one of these three.  Hence, to lay the groundwork for the proof of Theorem \ref{theorem-m-equals-3}, we first prove three technical lemmas which compute the chromatic numbers in these three circumstances.  The proofs of all three use the same main idea: Add or subtract rows to map to a graph with a $2\times 2$ matrix.  If such maps fail to produce a $3$-coloring, then show that $X$ contains a diamond lanyard.

We begin with ``first column all ones'' matrices.

\begin{Lem}\label{lemma-first-column-pos-1-neg-1-pos-1}Suppose we have \[\begin{pmatrix}
    1 && 0 \\
    1 && y_{22} \\
    1 && y_{32} 
    \end{pmatrix}^{\text{SACG}}_X.\]  Then $X$ has loops if and only if $\{y_{22},y_{32}\}$ is $\{0,-1\}$ or $\{0,1\}$ or $\{-1\}$ or $\{1\}$.   Otherwise, \[\chi(X)=\begin{cases}
        3 & \text{ if }y_{32}\equiv -y_{22}\Mod 3\\
        4 & \text{ if }y_{32}\not \equiv -y_{22}\Mod 3.
    \end{cases}\]\end{Lem}

\begin{proof}The statement about loops is straightforward.  Now suppose that $X$ does not have loops.  Let $M_X$ be the matrix in the lemma statement.

The first column's sum is odd, so $X$ cannot be bipartite, by \cite[Lemma 2.11]{Cervantes-Krebs-general-publilshed}.

If $y_{32}\equiv -y_{22}\Mod 3$, then we get $\chi(X)=3$ by applying \cite[Lemma 2.12]{Cervantes-Krebs-general-publilshed}.

Now assume that $y_{32}\not \equiv -y_{22}\Mod 3$, in other words that $y_{32}=-y_{22}\pm 1+3\ell$ for some integer $\ell$.  First we show that $\chi(X)>3$, by showing that $X$ contains a diamond lanyard.  Let $H$ be the subgroup of $\mathbb{Z}^3$ generated by the columns of $M_X$.  Recall that vertices of $X$ are of the form $(a,b,c)^t+H$, which we denote by $\overline{(a,b,c)^t}$.  The vertices $\overline{(0,0,0)^t},\overline{(0,1,0)^t},\overline{(0,0,-1)^t},$ and $\overline{(0,1,-1)^t}$ form a diamond in $X$.  Shifting this diamond $x-1$ times by $\overline{(0,1,-1)^t}$ and concatenating, we produce an unclasped diamond lanyard $L_1$ of length $x$ with endpoints $\overline{(0,0,0)^t}$ and $\overline{(0,x,-x)^t}$.  In a similar vein, we have an unclasped diamond lanyard of length $2$ formed by one diamond with vertices $\overline{(0,0,0)^t}, \overline{(0,1,0)^t}, \overline{(0,1,1)^t},$ and $\overline{(0,2,1)^t}$ and another diamond with vertices $\overline{(0,2,1)^t}, \overline{(0,2,0)^t}, \overline{(0,3,1)^t},$ and $\overline{(0,3,0)^t}$.  Its endpoints are $\overline{(0,0,0)^t}$ and $\overline{(0,3,0)^t}$.  Assume for now that $\ell\geq 1$.  Shifting $\ell-1$ times by $\overline{(0,3,0)^t}$ and concatenating, we produce an unclasped diamond lanyard $L_2$ of length $2\ell$ with endpoints $\overline{(0,0,0)^t}$ and $\overline{(0,3\ell,0)^t}$.  Conjoining $L_1$ and $L_2$ gives us an unclasped diamond lanyard of length $x+2\ell$ with endpoints $\overline{(0,x,-x)^t}$ and $\overline{(0,3\ell,0)^t}$.  Taking $x=y_{32}$ produces an edge between $\overline{(0,x,-x)^t}$ and $\overline{(0,3\ell,0)^t}$, and thus we have a clasped diamond lanyard in $X$.  By Lemma \ref{lemma-chi-of-diamond-lanyard}, we have that $\chi(X)\geq 4$.  A similar procedure gives the same result when $\ell\leq 0$.

% (Observe that all triples in the preceding argument have $0$ as their first coordinate.  Indeed, because of the first column of $M_X$ we have that $\overline{(a,b,c)^t}=\overline{(0,b-a,c-a)^t}$.  Hence, only the last two coordinates of each vertex, we may visualize our diamond lanyard in the plane.)

Now we show that $\chi(X)\leq 4$.  By \cite[Lemma 2.10]{Cervantes-Krebs-general-publilshed} we have a homomorphism

\[\begin{pmatrix}
    1 && 0 \\
    1 && y_{22} \\
    1 && y_{32} 
    \end{pmatrix}_X^{\text{SACG}}\xrightarrow{\ocirc}
    \begin{pmatrix}
    1 && 0 \\
    2 && y_{22}+y_{32} 
    \end{pmatrix}_Y^{\text{SACG}}.\]
    
The results of \S\ref{subsection-m-2} imply that we fail to get a $4$-coloring if and only if $y_{22}+y_{32}\in\{-5, -2, -1, 1, 2, 5\}$.

First we consider the case where $y_{22}+y_{32}=-5$.  Then \[\begin{pmatrix}
    1 && 0 \\
    1 && y_{22} \\
    1 && -y_{22}-5 
    \end{pmatrix}_X^{\text{SACG}}\xrightarrow{\ocirc}
    \begin{pmatrix}
    2 && -y_{22}-5 \\
    1 && y_{22} 
    \end{pmatrix}^{\text{SACG}}\cong \begin{pmatrix}
    1 && 0 \\
    2 && 3y_{22}+5 
    \end{pmatrix}_{Y'}^{\text{SACG}}
    .\]

Here we use \cite[Lemmas 2.10 and 2.6]{Cervantes-Krebs-general-publilshed}, and from now on we shall do this sort of thing without referring to these lemmas each time.  By Theorem \ref{theorem-m-equals-2}, we have that $Y'$ is $4$-colorable unless $y_{22}=-1$ or $y_{22}=0$.  But in the first case we have:\[\begin{pmatrix}
    1 && 0 \\
    1 && -1 \\
    1 && -4 
    \end{pmatrix}_X^{\text{SACG}}\xrightarrow{\ocirc}\begin{pmatrix}
    2 && -1 \\
    1 && -4 
    \end{pmatrix}^{\text{SACG}}\cong C_7(1,4),\]which is $4$-colorable.  In the second case, we have a graph homomorphism to $C_{10}(1,2)$, which is which is $4$-colorable.  Here we use Theorems \ref{theorem-m-equals-2} and \ref{theorem-Heubergers} as well as Lemma \ref{lemma-sufficient-condition-to-be-properly-given-circulant}.  In the sequel, usually we will simply compute chromatic numbers of graphs with $2\times 2$ matrices using the results of \S\ref{subsection-m-2} without referring to the specific theorems and lemmas used.

The other cases, where $y_{22}+y_{32}$ equals $-2$ or $-1$ or $1$ or $2$ or $5$, can each be dealt with in a similar way.  We omit the proofs here, but complete details can be found in our authors' notes, which are housed on the second author's website \cite{Cervantes-Krebs}.\end{proof}

\vspace{.1in}

Next we tackle ``L-shaped'' matrices.  By performing row and column operations not unlike those in \S\ref{subsection-restricted-Hermite-normal-form}, it suffices to consider only matrices with some additional restrictions imposed.

\begin{Lem}\label{lemma-L-shaped}Suppose we have \[\begin{pmatrix}
    y_{11} && 0 \\
    y_{21} && 0 \\
    y_{31} && y_{32} 
    \end{pmatrix}^{\text{SACG}}_X,\]  
    
    where $y_{11},y_{21},y_{32}>0$ and $-\frac{y_{32}}{2}\leq y_{31}\leq 0$.
    
    Then:
    
\begin{enumerate}

    \item We have that $X$ has loops if and only if $y_{32}=1$.
    
    \item We have that $\chi(X)=2$ if and only if $y_{11}+y_{21}+y_{31}$ and $y_{32}$ are both even.
    
    \item We have that $\chi(X)=4$ if and only if $y_{11}=y_{21}=-y_{31}=1$ and $3\nmid y_{32}$ and $y_{32}>1$.
    
    \item Otherwise, $\chi(X)=3$.

\end{enumerate}
\end{Lem}

\begin{proof}The first two statements are straightforward to prove.  Now suppose that $X$ does not have loops (i.e., that $y_{32}\geq 2$) and is not bipartite.  Let $M_X$ be the matrix in the lemma statement.  We have that $3$ divides either $y_{11}$, $y_{21}$, $y_{11}+y_{21}$, or $y_{11}-y_{21}$, and we divvy into cases accordingly.

First suppose that $3\mid y_{11}$.  We have\[\begin{pmatrix}
    y_{11} && 0 \\
    y_{21} && 0 \\
    y_{31} && y_{32} 
    \end{pmatrix}^{\text{SACG}}_X\xrightarrow{\ocirc}\begin{pmatrix}
    y_{11} && 0 \\
    y_{21}+y_{31} && y_{32} 
    \end{pmatrix}^{\text{SACG}}_Y.\]We see that $3\mid\det M_Y$, where $M_Y$ is the Heuberger matrix for $Y$ shown above.  By Cor. \ref{corollary-det-divis-by-3}, it follows that $Y$ is $3$-colorable unless it has loops.  But this cannot happen, because $3\mid y_{11}$ and $y_{11}>0$ and $y_{32}\geq 2$.

The case where $3\mid y_{21}$ is handled similarly, as is the case where $3\mid y_{11}+y_{21}$; in this latter case, begin with a homomorphism that ``collapses'' the top two rows by adding them.

Finally, suppose that $3\mid y_{11}-y_{21}$.  We may assume without loss of generality that $y_{11}\geq y_{21}$, for if not, then swap the top two rows.  We have\[\begin{pmatrix}
    y_{11} && 0 \\
    y_{21} && 0 \\
    y_{31} && y_{32} 
    \end{pmatrix}^{\text{SACG}}_X\xrightarrow{\ocirc}\begin{pmatrix}
    y_{11}-y_{21} && 0 \\
    y_{31} && y_{32} 
    \end{pmatrix}^{\text{SACG}}_Y.\]

By Cor. \ref{corollary-det-divis-by-3}, it follows that $Y$ is $3$-colorable unless it has loops.  Because $y_{32}\geq 2$ and $3\mid y_{11}-y_{21}$, this occurs if and only if $y_{11}-y_{21}=0$ and $\text{gcd}(y_{31}, y_{32})=1$.

Now suppose that (ii) occurs.  Let $\alpha, \beta\in\mathbb{Z}$ such that $\alpha y_{31}+\beta y_{32}=1$.  Multiply $M_X$ on the right by the unimodular matrix \[\begin{pmatrix}
    \alpha & -y_{32}\\
    \beta & y_{31}
\end{pmatrix}\] to get \[\begin{pmatrix}
    y_{11} & 0\\
    y_{11} & 0\\
    y_{31} & y_{32}
\end{pmatrix}^{\text{SACG}}_X=\begin{pmatrix}
    \alpha y_{11} & -y_{11}y_{32}\\
    \alpha y_{11} & -y_{11}y_{32}\\
    1 & 0
\end{pmatrix}^{\text{SACG}}_X\xrightarrow{\ocirc}\begin{pmatrix}
    1 & 0\\
    -2\alpha y_{11} & 2y_{11}y_{32}
\end{pmatrix}^{\text{SACG}}_{Y'}\]

Provided $Y'$ does not have loops, $Y'$ is isomophic to the Heuberger circulant $C_{n'}(a',b')$ with $n'=2y_{11}y_{32}$, $a'=2\alpha y_{11}$, $b'=1$.  So $Y'$ is $3$-colorable unless it has loops or one of the exceptional cases in Theorem \ref{theorem-Heubergers} occurs.  We now deal with these possibilities one at a time.

Suppose $Y'$ has loops.  Because $y_{11}>0$ and $y_{32}\geq 2$, this occurs if and only if $2y_{11}y_{32}\mid 2\alpha y_{11}$, which happens if and only if $y_{32}\mid \alpha$.  But then using column operations, we get that $X$ has loops, contrary to our assumptions.  Observe that $n'=\pm 5$ and $n'=\pm13$ cannot happen, because $n'$ is even.  In the remaining cases, we show $y_{11}=1$, then we proceed from there.  Suppose $a'\equiv \pm 2b'\Mod{n'}$. So $2\alpha y_{11}\equiv \pm 2\Mod{2y_{11}y_{32}}$.  So $\alpha y_{11}\equiv \pm 1\Mod{y_{11}y_{32}}$.  So $y_{11}\mid \pm 1$, which implies $y_{11}=1$.  Suppose $2a'\equiv \pm b'\Mod{n}$. So $4\alpha y_{11}\equiv \pm 1\Mod{2y_{11}y_{32}}$.  So $y_{11}\mid \pm 1$, which implies $y_{11}=1$.

Thus we have:\[\begin{pmatrix}
    1 & 0\\
    1 & 0\\
    y_{31} & y_{32}
\end{pmatrix}^{\text{SACG}}_X\xrightarrow{\ocirc}\begin{pmatrix}
    1 & 0\\
    y_{31}+1 & y_{32}
\end{pmatrix}^{\text{SACG}}_{Y''}\]

Then $Y''$ has loops if and only if $y_{32}\mid y_{31}+1$.  But then $y_{31}=ky_{32}-1$ for some integer $k$.  From our assumption that $-\frac{y_{32}}{2}\leq y_{31}\leq 0$, we must have that $k=0$ and $y_{31}=-1$.  The conclusion of the lemma now follows in this case from Lemma \ref{lemma-first-column-pos-1-neg-1-pos-1}.

So now we can assume that $y_{31}<-1$.  Then by Lemma \ref{lemma-sufficient-condition-to-be-properly-given-circulant}, we have that $Y''$ is isomorphic to the Heuberger ciculant $C_{n''}(a'',b'')$ with $n''=y_{32}$, $a''=-y_{31}-1$, $b''=1$.  Suppose $Y''$ is not $3$-colorable.  So one of the exceptional cases in Theorem \ref{theorem-Heubergers} occurs.  We now deal with these possibilities one at a time.

\guillemotright\, Suppose that $n''=5$ and $a''\equiv \pm 2b''\Mod{5}$. Then $y_{32}=5$ and $a''\equiv 2$ or $3$ mod $5$.  Hence $y_{31}=-2$, using that $-\frac{y_{32}}{2}\leq y_{31}\leq 0$ as well as that $y_{31}<-1$.   But then \[\begin{pmatrix}
    1 & 0\\
    1 & 0\\
    -2 & 5
\end{pmatrix}^{\text{SACG}}_X\xrightarrow{\ocirc}\begin{pmatrix}
    1 & 0\\
    -1 & 5
\end{pmatrix}^{\text{SACG}}\] gives us a $3$-coloring of $X$.

\guillemotright\, Suppose $n''=13$ and one of $a''$ or $b''$ is congruent to $\pm 5$ times the other modulo $13$: Then $y_{32}=13$ and $a''\equiv 5$ or $8$ modulo $13$.  From $-\frac{y_{32}}{2}\leq y_{31}\leq 0$ and $y_{31}<-1$ and $a''=-y_{31}-1$, we get $y_{31}=-6$.  But then\[\begin{pmatrix}
    1 & 0\\
    1 & 0\\
    -6 & 13
\end{pmatrix}^{\text{SACG}}_X\xrightarrow{\ocirc}\begin{pmatrix}
    2 & 0\\
    -6 & 13
\end{pmatrix}^{\text{SACG}}\xrightarrow{\ocirc}\begin{pmatrix}
    1 & 0\\
    -3 & 13
\end{pmatrix}^{\text{SACG}}\]gives us a map to a $3$-colorable Heuberger circulant.

\guillemotright\, Suppose $a''\equiv 2b''\Mod{n''}$.  Then $y_{32}\mid y_{31}+3$.  Note we cannot have $y_{31}=-2$, since then $y_{32}\mid 1$, but $y_{32}>1$.

If $y_{31}=-3$, then we have \[\begin{pmatrix}
    1 & 0\\
    1 & 0\\
    -3 & y_{32}\end{pmatrix}^{\text{SACG}}_X\xrightarrow{\ocirc}\begin{pmatrix}
        2 & 0\\
        -3 & y_{32}
    \end{pmatrix}^{\text{SACG}}_{Z}
\]

The graph $Z$ cannot have loops.  Let $a=3, b=2, n=2y_{32}$.  Then $Z$ is isomorphic to the Heuberger circulant $C_n(a,b)$.  We cannot have $n=5$ or $n=13$, because $n$ is even.  If $a\equiv 2b\Mod{n}$, then $2y_{32}\mid -1$, which cannot happen.  If $a\equiv -2b\Mod{n}$, then $2y_{32}\mid 7$, which cannot happen.  If $2a\equiv b\Mod{n}$, then $2y_{32}\mid 4$, which implies that $y_{32}=2$, but this violates $-\frac{y_{32}}{2}\leq y_{31}\leq 0$.  If $2a\equiv -b\Mod{n}$, then $2y_{32}\mid 8$, which implies that $y_{32}=2$ or $y_{32}=4$, both of which violate $-\frac{y_{32}}{2}\leq y_{31}\leq 0$.

Now assume $y_{31}<-3$, which implies that $y_{31}+3< 0$.  So from $y_{32}\mid y_{31}+3$, we get that $y_{32}\leq -3-y_{31}\leq -3+\frac{y_{32}}{2}$.  But $y_{32}>0$, so this cannot happen.

\guillemotright\, Suppose that $a''\equiv -2b''\Mod{n''}$.  Then $y_{32}\mid y_{31}-1$.  So $y_{32}\leq 1-y_{31}\leq 1+\frac{y_{32}}{2}$.  But $y_{32}>1$.

\guillemotright\, Suppose that $2a''\equiv b''\Mod{n''}$.  Then $y_{32}\mid 2y_{31}+3$.  Because $y_{31}\leq -2$, we have $2y_{31}+3<0$.  Hence $y_{32}\leq -2y_{31}-3\leq y_{32}-3$, which is a contradiction.

\guillemotright\, Suppose that $2a''\equiv -b''\Mod{n''}$.  Then $y_{32}\mid 2y_{31}+1$.  So $y_{32}\leq -2y_{31}-1\leq y_{32}-1$, which is a contradiction.\end{proof}

\vspace{.05in}

In our final preparatory step, we contemplate ``$I$ on top'' matrices.  As before, we obtain upper bounds by mapping to graphs with $2\times 2$ matrices, and we obtain a lower bound in some cases by finding diamond lanyards as subgraphs.

\begin{Lem}\label{lemma-I-on-top}Suppose $X$ is a standardized abelian Cayley graph with an associated Heuberger matrix $$M_X=\begin{pmatrix}
1 & 0\\
0 & 1\\
y_{31} & y_{32}
\end{pmatrix},$$ where $y_{31},y_{32}>0$ and $y_{31}\leq y_{32}$.  Then \[\chi(X)=\begin{cases}
    2 & \text{if } y_{31} \text{ and } y_{32} \text{ are both odd}\\
    4 & \text{if } y_{31}=2\text{ and }3\mid y_{32}\\
    4 & \text{if } 1\not\equiv y_{31}\Mod 3\text{ and }y_{32}=1+y_{31}\\
    3 & \text{otherwise}.
    
    \end{cases}\]\end{Lem}

Before embarking on the proof, we note that by \cite[Example 2.4]{Cervantes-Krebs-general-publilshed}, we have that $X$ is isomorphic to the distance graph $\text{Cay}(\mathbb{Z},\{\pm 1, \pm y_{31}, \pm y_{32}\})$.  Hence Lemma \ref{lemma-I-on-top} is a special case of Zhu's theorem, as discussed in \cite{Cervantes-Krebs-Zhu}.  We offer here an alternative proof of Zhu's theorem using Heuberger matrices.

\begin{proof}
\cite[Lemma 2.11]{Cervantes-Krebs-general-publilshed} implies that $\chi(X)=2$ if and only if $y_{31}$ and $y_{32}$ are both odd.

To show that $X$ is $4$-colorable, consider

\[\begin{pmatrix}
    1 & 0\\
    0 & 1\\
    y_{31} & y_{32}\end{pmatrix}^{\text{SACG}}_X\xrightarrow{\ocirc}\begin{pmatrix}
        1 & 0\\
        y_{31} & 1+y_{32}
    \end{pmatrix}^{\text{SACG}}_{Y}\]

By Lemma \ref{lemma-loops-2-by-2} and \ref{lemma-sufficient-condition-to-be-properly-given-circulant}, we see that $Y$ does not contain loops and is isomorphic to the Heuberger circulant $C_{1+y_{32}}(1, y_{31})$.  So $Y$ is $4$-colorable unless $y_{32}=4$ and $y_{31}\in\{2,3\}$.  But in these cases respectively take \[\begin{pmatrix}
    1 & 0\\
    0 & 1\\
    2 & 4\end{pmatrix}^{\text{SACG}}_X\xrightarrow{\ocirc}\begin{pmatrix}
        1 & 0\\
        2 & 3
    \end{pmatrix}^{\text{SACG}} \text{ and }\begin{pmatrix}
    1 & 0\\
    0 & 1\\
    3 & 4\end{pmatrix}^{\text{SACG}}_X\xrightarrow{\ocirc}\begin{pmatrix}
        1 & -1\\
        3 & 4
    \end{pmatrix}^{\text{SACG}}.\]

Indeed, $Y$ is $3$-colorable unless one of the six exceptional cases in Theorem \ref{theorem-Heubergers} occurs.  Those cases each place restrictions on $y_{31}$ and $y_{32}$, whereupon we can modify the mapping appropriately to try to get a $3$-coloring.  One can show that this procedure will produce a $3$-coloring unless either $y_{31}=2$ and $3\mid y_{32}$, or else $1\not\equiv y_{31}\Mod 3$ and $y_{32}=1+y_{31}$.  The logic is quite similar to that in the proofs of Lemmas \ref{lemma-first-column-pos-1-neg-1-pos-1} and \ref{lemma-L-shaped}, so we omit it here.  Complete details can be found in our authors' notes, which are housed on the second author's website \cite{Cervantes-Krebs}.

Finally, we show that if either $y_{31}=2$ and $3\mid y_{32}$, or else $1\not\equiv y_{31}\Mod 3$ and $y_{32}=1+y_{31}$, then $X$ contains a diamond lanyard.  By Lemma \ref{lemma-chi-of-diamond-lanyard}, this will show that $\chi(X)\geq 4$ in these cases.  We note that this is essentially what Zhu does in \cite{Zhu} to find a lower bound on the fractional chromatic number of distance graphs such as these.  Let $H$ be the subgroup of $\mathbb{Z}^3$ generated by the columns of $M_X$.  We denote by $\overline{(a,b,c)^t}$ the vertex $(a,b,c)^t+H$ of $X$.

Suppose $y_{31}=2$ and $3\mid y_{32}$.  We have that $y_{32}=3k\pm1$ for some positive integer $k$.
There is a diamond in $X$ with vertices $\overline{(0,0,0)^t}, \overline{(0,0,1)^t}, \overline{(0,0,2)^t}$, and $\overline{(0,0,3)^t}$.  
Shifting this $k-1$ times by $\overline{(0,0,3)^t}$ and concatenating, we obtain a diamond lanyard of length $k$ with endpoints $\overline{(0,0,0)^t}$ and $\overline{(0,0,3k)^t}$.

Now suppose that $1\not\equiv y_{31}\Mod 3$ and $y_{32}=1+y_{31}$.  So either $y_{31}$ or $y_{32}$ equals $3k$ for some positive integer $k$.  We have a diamond in $X$ with vertices $\overline{(0,0,0)^t}, \overline{(0,0,1)^t}, \overline{(0,0,y_{32})^t}$, and $\overline{(0,0,y_{32}+1)^t}$.  Shifting this by $\overline{(0,0,y_{32}+1)^t}$, we obtain an unclasped diamond lanyard of length two with endpoints $\overline{(0,0,0)^t}$ and $\overline{(0,0,2y_{32}+2)^t}$.  Append to this a diamond with vertices $\overline{(0,0,2y_{32}+2)^t}, \overline{(0,0,y_{32}+3)^t}, \overline{(0,0,y_{32}+2)^t}$, and $\overline{(0,0,3)^t}$.  We thus obtain an unclasped diamond lanyard of length three with endpoints $\overline{(0,0,0)^t}$ and $\overline{(0,0,3)^t}$.  Shifting this $k-1$ times by $\overline{(0,0,3)^t}$ and concatenating, we obtain a diamond lanyard of length $3k$ with endpoints $\overline{(0,0,0)^t}$ and $\overline{(0,0,3k)^t}$. \end{proof}

Finally, we turn our attention to proving Theorem \ref{theorem-m-equals-3}.  The essense of the proof is to show that if we do not have a homomorphism from $X$ to a $3$-colorable graph with a $2\times 2$ matrix, then $M_X$ must be in a form where (perhaps after some manipulations) one of the preceding three lemmas applies.

\begin{proof}[Proof of Theorem \ref{theorem-m-equals-3}]The first statement follows from Lemma \ref{lemma-loops}, and the second follows from \cite[Lemma 2.11]{Cervantes-Krebs-general-publilshed}.

Now suppose that $M_X$ is one of the six types of matrices listed in the third statement.  Lemma \ref{lemma-I-on-top} shows that if $M_X$ is of the form \(\begin{pmatrix}
1 & 0\\
0 & 1\\
3k & 1+3k
\end{pmatrix}\), then $\chi(X)=4$.  For the other five, we can perform row and column operations as per \cite[Lemma 2.6]{Cervantes-Krebs-general-publilshed} to obtain a matrix for an isomorphic graph so that either Lemma \ref{lemma-I-on-top} or \ref{lemma-first-column-pos-1-neg-1-pos-1} proves the third statement of the corollary.  For example, suppose $M_X$ is of the form \(\begin{pmatrix}
1 & 0\\
-1 & 2\\
-1-3k & 2+3k
\end{pmatrix}\) for an integer $k$.  Add the second column to the first and then multiply the third row by $-1$ to produce the matrix \(\begin{pmatrix}
1 & 0\\
1 & 2\\
1 & -2-3k
\end{pmatrix},\) whereupon Lemma \ref{lemma-first-column-pos-1-neg-1-pos-1} applies.  We leave the computations in the other four cases to the reader.

Finally, assume that none of the first three statements apply.  We will show that $X$ has a $3$-coloring.

Take the following mapping:\[\begin{pmatrix}
    y_{11} & 0\\
    y_{21} & y_{22}\\
    y_{31} & y_{32}
\end{pmatrix}^{\text{SACG}}_X\xrightarrow{\ocirc} \begin{pmatrix}
    y_{11} & 0\\
    y_{21}-y_{31} & y_{22}-y_{32}
\end{pmatrix}^{\text{SACG}}_Y\]

Let $M_Y$ be the $2\times 2$ matrix given above for $Y$.  From Def. \ref{def-modified-Hermite-normal-form} we have that $3\mid \det M_Y$.  So by Cor. \ref{corollary-det-divis-by-3}, we have that $Y$ (and hence $X$) is $3$-colorable unless $Y$ has loops.  (We remark that we imposed the third condition in Def. \ref{def-modified-Hermite-normal-form} specifically so that we can use Cor. \ref{corollary-det-divis-by-3} right here.)  By Lemma \ref{lemma-loops-2-by-2} and Def. \ref{def-modified-Hermite-normal-form}, we have that either (i) $y_{22}-y_{32}=-1$, or (ii) $y_{11}=1$ and $y_{22}-y_{32}\mid y_{21}-y_{31}$.

Suppose (i) holds.  Because $3\mid \det M_Y$, we must have that $3\mid y_{11}$.  Now consider \[\begin{pmatrix}
    y_{11} & 0\\
    y_{21} & y_{22}\\
    y_{31} & y_{32}
\end{pmatrix}^{\text{SACG}}_X\xrightarrow{\ocirc} \begin{pmatrix}
    y_{11} & 0\\
    y_{21}+ y_{31} & y_{22}+y_{32}
\end{pmatrix}^{\text{SACG}}\]

By Cor. \ref{corollary-det-divis-by-3}, this produces a $3$-coloring unless the target graph has loops.  But by Lemma \ref{lemma-loops-2-by-2}, this occurs if and only if $y_{22}+y_{32}=\pm 1$, which gives us that $y_{22}=0$ and $y_{32}=1$.  But then the first statement in the theorem holds, contrary to assumption.

Thus (ii) holds.  Because $3\mid \det M_Y$, we must have $y_{32}=y_{22}+3k$ for some integer $k\geq 0$.  (Here we use that $y_{32}\geq y_{22}$.)  Also $\ell(y_{22}-y_{32})=y_{21}-y_{31}$ for some integer $\ell$, which gives us that $y_{31}=y_{21}+3k\ell$.  So we have \begin{equation}\label{equation-MX}
M_X=\begin{pmatrix}
    1 & 0\\
    y_{21} & y_{22}\\
    y_{21}+3k\ell & y_{22}+3k
\end{pmatrix}.\end{equation}

% FOR $2\times 2$ MATRICES, IF THEY ARE NOT $3$-COLORABLE, THEN WE HAVE THAT EITHER THE DETERMINANT IS NOT DIVISIBLE BY $3$, OR THERE ARE LOOPS

Take the mapping
\[
    \begin{pmatrix}
    1 & 0 \\
    y_{21} & y_{22} \\
    y_{21}+3k\ell & y_{22}+3k
    \end{pmatrix}^{\text{SACG}}_X
    \xrightarrow{\ocirc}
    \begin{pmatrix}
    1 & 0\\
    2y_{21}+3k\ell & 2y_{22}+3k
    \end{pmatrix}^{\text{SACG}}_{Y'}.
\]

Either $Y'$ has loops, or else by Lemma \ref{lemma-sufficient-condition-to-be-properly-given-circulant} we have that $Y'$ is isomorphic to the Heuberger circulant $C_{n'}(a',b')$ with $n'=2y_{22}+3k$ and $a'=-2y_{21}-3k\ell$ and $b'=1$.

\vspace{.05in}

\guillemotright\, First suppose that $Y'$ has loops.  By Lemma \ref{lemma-loops-2-by-2}, this occurs if and only if $2y_{22}+3k\mid 2y_{21}+3k\ell$.  Then $2y_{21} + 3k\ell = (2y_{22} + 3k)q$ for some $q \in \Z$. So $y_{21} = qy_{22} + \frac32 k(q - \ell)$. Letting $t = q - \ell$, by various columm operations we have
\begin{align*}
    \begin{pmatrix}
    1 && 0 \\
    y_{21} && y_{22} \\
    y_{21} + 3k\ell && y_{22} +3k
    \end{pmatrix}^{\text{SACG}}_{X}
    & =
    \begin{pmatrix}
    1 && 0 \\
    qy_{22} + \frac32 k(q - \ell) && y_{22} \\
    qy_{22} + \frac32 k(q + \ell) && y_{22} +3k
    \end{pmatrix}^{\text{SACG}}_{X}
    =
    \begin{pmatrix}
    1 && 0 \\
    \frac32 k(q - \ell) && y_{22} \\
    \frac32 k(-q + \ell) && y_{22} +3k
    \end{pmatrix}^{\text{SACG}}_{X}
    \\
    & =
    \begin{pmatrix}
    1 && 0 \\
    \frac32 kt && y_{22} \\
    - \frac32 kt && y_{22} +3k
    \end{pmatrix}^{\text{SACG}}_{X}
    \xrightarrow{\ocirc}
    \begin{pmatrix}
    \frac32 kt && y_{22} \\
    -1 - \frac32 kt && y_{22} +3k
    \end{pmatrix}^{\text{SACG}}_{Y''}
\end{align*}

Either $Y''$ has loops, or else by Lemma \ref{lemma-sufficient-condition-to-be-properly-given-circulant} we have that $Y''$ is isomorphic to the Heuberger circulant $C_{n''}(a',b')$ with $n''=(3kt + 1) y_{22} + \frac92 k^2t$ and $a''=\frac32 kt+1$ and $b''=\frac32 kt$.

\vspace{.05in}

\guillemotright\, \guillemotright\, First suppose $Y''$ has loops. Let $M_{Y''}$ be the given matrix for $Y''$.  By Lemma \ref{lemma-loops-2-by-2} we have that either the top or bottom row of $M_{Y''}$ is zero, or else $n''$ divides every entry in a row of $M_{Y''}$.  If the first row is zero, then $kt = 0$ and $X$ had loops to begin with, contrary to assumption. If the second row is zero, then we have $-1 - \frac32 kt = 0$ and $y_{22}+3k=0$, so
\[
    \begin{pmatrix}
    1 && 0 \\
    \frac32 kt && y_{22} \\
    - \frac32 kt && y_{22} +3k
    \end{pmatrix}^{\text{SACG}}_{X}
    =
    \begin{pmatrix}
    1 && 0 \\
    -1 && -3k \\
    1 && 0
    \end{pmatrix}^{\text{SACG}}_{X}
    \cong 
    \begin{pmatrix}
    1 && 0 \\
    1 && 3k \\
    1 && 0
    \end{pmatrix}^{\text{SACG}},\]which is $3$-colorable by Lemma \ref{lemma-first-column-pos-1-neg-1-pos-1}.

Now suppose that $n''$ divides every entry in either the first or second row of $M_{Y''}$.  We will work out here the details of the former case; for the latter, which is similar, see the authors' notes at \cite{Cervantes-Krebs}.  We have that $n'' \mid \frac32 kt$ and $n''\mid y_{22}$.  Observe that because $X$ does not have loops, hence $t \neq 0$ and $k\neq 0$ (and therefore $k>0$).  We split into cases according to whether $t$ is positive or negative.

First suppose $t>0$.  From $n'' \mid \frac32kt$ we get
\begin{gather*}
    -\frac32kt \leq(3kt+1)y_{22} + \frac92k^2t \leq \frac32kt \\
\;\\    
\frac{-\frac32kt - \frac92k^2t}{3kt+1} \leq y_{22} \leq \frac{\frac32kt - \frac92k^2t}{3kt+1} \qquad \text{because }t>0\\
\;\\
-\frac12-\frac32 k< y_{22} \leq -\frac12-\frac32 k+\frac{-\frac12 +\frac32 k}{3kt+1}<-\frac32 k
\end{gather*}This cannot happen, because both $k$ and $y_{22}$ are integers.

Now suppose $t<0$.  From $n'' \mid \frac32kt$ after some calculations we get that \[\frac12 - \frac32k > y_{22} \geq -\frac32 - \frac32k.\]
Using the fact that $y_{22}$ and $k$ are integers, this tells us that $y_{22}=-\frac32 k+\epsilon$ where $\epsilon\in\{0, -\frac12, -1, -\frac32\}$.  We will work out here only the cases where $\epsilon=0$ and $\epsilon=-1$; the other two cases are similar, and details can be found at \cite{Cervantes-Krebs}.

\guillemotright\, \guillemotright\, \guillemotright\, Suppose $\epsilon=0$.  Then we have
\begin{equation*}
    \begin{pmatrix}
        1 & 0 \\
        \frac32kt & y_{22} \\
        -\frac32kt & y_{22}+3k
    \end{pmatrix}^{\text{SACG}}_X
    =
    \begin{pmatrix}
        1 & 0 \\
        \frac32kt & -\frac32k \\
        -\frac32kt & \frac32k
    \end{pmatrix}^{\text{SACG}}_X
    =
    \begin{pmatrix}
        1 & 0 \\
        0 & -\frac32k \\
        0 & \frac32k
    \end{pmatrix}^{\text{SACG}}_X
.\end{equation*}
But this would mean that $X$ has loops, contrary to assumption.

\guillemotright\, \guillemotright\, \guillemotright\, Suppose $\epsilon=-1$.
So we must have $y_{22} = -1 - \frac32k$, which gives us that $n''=-3kt - \frac32k -1$.  Because $y_{22}\in\mathbb{Z}$ and $k>0$, it follows that $k$ must be even, so $y_{22}\leq -4$.  From $n'' \mid y_{22}$, we get that \[-1-\frac32 k\leq n''\leq 1+\frac32 k,\] which implies that $0\geq t\geq -\frac53$ and hence $t=-1$.  Recall that \[y_{21} = qy_{22} + \frac32 k(q - \ell) = qy_{22}+\frac32 kt=qy_{22}-\frac32 k=(q+1)y_{22}+1.\]By Def. \ref{def-modified-Hermite-normal-form}, we have that $\frac{y_{22}}{2}=-\frac{|y_{22}|}{2}\leq y_{21}\leq 0$.  Thus we have $y_{22}\leq (2q+2)y_{22}+2\leq 0$.  Solving for $q$ and using that $y_{22}\leq -4$, we find that $-\frac14\geq q>-1$, which is a contradiction, since $q$ is an integer.

\vspace{.1in}

\guillemotright\, \guillemotright\, Now suppose that $Y''$ does not have loops, so that $Y''$ is isomorphic to the Heuberger circulant $C_{n''}(a'',b'')$.  We have that $Y''$ (and hence $X$) is $3$-colorable unless one of the six exceptional cases in Thm. \ref{theorem-Heubergers} occurs.  Here we go through only the case where $b''\equiv -2a''\Mod{n''}$ and $3\nmid n''$, and even then, only a few of the illustrative sub-cases.  The others can be found at \cite{Cervantes-Krebs}---as we go along, we will no longer mention that every time.

The condition $b''\equiv -2a''\Mod{n''}$ implies that $n''\mid \frac92 kt+2$.  We split into cases according to whether $t$ is positive or negative, and we discuss here only the case $t<0$.  We have \[\frac92 kt+2\leq (3kt+1)y_{22}+\frac92 k^2t\leq -\frac92 kt-2.\]  After some algebra this leads to \[-\frac32 k+\frac32>y_{22}\geq -\frac32 k-2.\]

Because $3\nmid n''=(3kt+1)y_{22}+\frac92 k^2t$, we have that $3\nmid y_{22}$.  Hence, because $k$ and $y_{22}$ are integers, we have that $y_{22}=-\frac32 k+\epsilon$ where $\epsilon\in\{1, \frac12, -\frac12, -1, -2\}$.  We cover here only the cases $\epsilon=\frac12$ and $\epsilon=-\frac12$, as they illustrate some of the various possibilities that can arise.

\guillemotright\,\guillemotright\,\guillemotright\,Suppose $\epsilon=\frac12$.  It follows that $n''=-\frac32k(1-t)+\frac12$, from which we find that $k$ is odd and $t$ is even.  From $n''\mid \frac92 kt+2$ we get that $n''\mid \frac92 k+\frac12$ gives us that $n''=\pm 1$.  Thus \[\frac92 k+\frac12\leq -\frac32k(1-t)+\frac12.\]Solving for $t$, we find that \[-4-\frac{2}{3k}\leq t,\]which gives us that $t=-2$ or $t=-4$.  If $t=-2$, then $n''=-\frac92k+\frac12\leq -4$, which together with $n''\mid \frac92 k+\frac12$ gives us that $n''\mid 1$, a contradiction.  So $t=-4$.  Then $n''=-\frac{15}{2}k+\frac12\leq -7$.  Thus $n''$ divides $5(\frac92 k+\frac12)+3n''=4$, a contradiction.

\guillemotright\,\guillemotright\,\guillemotright\,Suppose $\epsilon=-\frac12$.  It follows that $n''=-\frac32k(1+t)-\frac12$.  Thus $k$ is odd and $t$ is even.  Using the same sort of logic as in the $\epsilon=\frac12$ case, we get that $n''\mid-\frac92k+\frac12$, and we again find that $t=-2$ or $t=-4$.  If $t=-2$, we have $n''=\frac32 k-\frac12>0$.  So $n''$ divides $-\frac92k+\frac12+3n''=-1$, giving us $n''=1$, in turn giving us $k=1$.  But then $y_{22}=-\frac32 k+\epsilon=-2$ and $y_{32}=y_{22}+3k=1$, violating the condition $|y_{22}|\leq |y_{32}|$ from Def. \ref{def-modified-Hermite-normal-form}.

So $t=-4$.  This gives us that $n''=\frac92 k-\frac12\geq 4$.  By the choice of $\epsilon$, we have $y_{22}=-\frac32 k-\frac12\leq -2$.  By Def. \ref{def-modified-Hermite-normal-form} and our previous formula for $y_{21}$, we have that \begin{equation}\label{equation-MHNF}
-\frac34 k-\frac14\leq q\left(-\frac32 k-\frac12\right)+\frac32 k(-4)\leq 0.    
\end{equation}Solving for $q$, we find that \begin{equation}\label{equation-solve-for-q}
-5\geq -7+\frac{8}{3k+1}\geq 2q\geq -8.\end{equation}Because $q$ is an integer, this gives us that $q=-3$ or $q=-4$.  If $q=-3$, then from (\ref{equation-solve-for-q}) and from the fact that $k$ is an odd positive integer, we get that $k=1$.  But then $y_{22}=-\frac32 k+\epsilon=-2$ and $y_{32}=y_{22}+3k=1$, violating the condition $|y_{22}|\leq |y_{32}|$ from Def. \ref{def-modified-Hermite-normal-form}.  So $q=-4$.  But this contradicts (\ref{equation-MHNF}).

\vspace{.05in}

\guillemotright\, Now suppose that $Y'$ does not have loops, so $Y'$ is isomorphic to the Heuberger circulant $C_{n'}(a',b')$ with $n'=2y_{22}+3k$ and $a'=-2y_{21}-3k\ell$ and $b'=1$.  We have that $Y'$ (and hence $X$) is $3$-colorable unless one of the six exceptional cases in Thm. \ref{theorem-Heubergers} occurs.  We work out here with some granularity only one of these cases, namely where $3\nmid n'$ and $a'\equiv 2b'\Mod{n'}$.  The other five cases are handled similarly.  Complete details can be found in our authors' notes, which are housed on the second author's website \cite{Cervantes-Krebs}.

From $3\nmid n'$ we get that $3\nmid y_{22}$. From $a'\equiv 2b'\Mod{n'}$ we get that $y_{21}=qy_{22}-1+\frac32 k(q-\ell)$ for some integer $q$.  Let $t=q-\ell$.  Note $k$ or $t$ must be even.  We have:

\[
    \begin{pmatrix}
    1 & 0 \\
    y_{21} & y_{22} \\
    y_{21}+3k\ell & y_{22}+3k
    \end{pmatrix}^{\text{SACG}}_X
    =\begin{pmatrix}
    1 & 0 \\
    qy_{22}-1+\frac32 k(q-\ell) & y_{22} \\
    qy_{22}-1+\frac32 k(q-\ell)+3k\ell & y_{22}+3k
    \end{pmatrix}^{\text{SACG}}_X\]
    
\begin{equation*}\label{equation-X-Case-3-with-t}=\begin{pmatrix}
    1 & 0 \\
    qy_{22}-1+\frac32 k(q-\ell) & y_{22} \\
    qy_{22}-1+\frac32 k(q+\ell) & y_{22}+3k
    \end{pmatrix}^{\text{SACG}}_X=\begin{pmatrix}
    1 & 0 \\
    -1+\frac32 k(q-\ell) & y_{22} \\
    -1+\frac32 k(-q+\ell) & y_{22}+3k
    \end{pmatrix}^{\text{SACG}}_X=\begin{pmatrix}
    1 & 0 \\
    -1+\frac32 kt & y_{22} \\
    -1-\frac32 kt & y_{22}+3k
    \end{pmatrix}^{\text{SACG}}_X\end{equation*}

Suppose $t=0$.  Then $y_{21}=qy_{22}-1$.  From Def. \ref{def-modified-Hermite-normal-form} we have that $-\frac{|y_{22}|}{2}\leq y_{21}\leq 0$.  Suppose $y_{22}>0$.  If $y_{22}=1$, then $y_{21}=0$, and in this case the theorem now follows from Lemma \ref{lemma-I-on-top}.  So we may assume that $y_{22}\geq 2$.  Then from $qy_{22}-1\leq 0$ we get that $q\leq 0$.  From $-\frac{y_{22}}{2}\leq qy_{22}-1$ we then get that $q=0$.  From $t=q-\ell$ we then get $\ell=0$.  So $y_{21}=y_{31}=-1$.  After multiplying the bottom row of $M_X$ by $-1$, the theorem now holds by Lemma \ref{lemma-first-column-pos-1-neg-1-pos-1}.  The same is true for similar reasons if $y_{22}<0$.  Hence we may assume that $t\neq 0$.  A similar argument shows that we may assume that $k\neq 0$.

We divide now into cases according to whether $t>0$ or $t<0$.  We write here only about the case $t<0$, the case $t>0$ being similar.

Consider the mapping\[\begin{pmatrix}
    1 & 0 \\
    -1+\frac32 kt & y_{22} \\
    -1-\frac32 kt & y_{22}+3k
    \end{pmatrix}^{\text{SACG}}_X\xrightarrow{\ocirc}\begin{pmatrix}
    \frac32 kt & y_{22} \\
    -1-\frac32 kt & y_{22}+3k
    \end{pmatrix}^{\text{SACG}}_{Y''}\]

Either $Y''$ has loops, or else by Lemma \ref{lemma-sufficient-condition-to-be-properly-given-circulant} we have that $Y''\cong C_{n''}(a'', b'')$ where $a''=1+\frac32 kt$ and $b''=\frac32 kt$ and $n''=(3kt+1)y_{22}+\frac92 k^2t$.  
Let $M_{Y''}$ be the Heuberger matrix for $Y''$ given above.  Because $kt\neq 0$ and $kt\in\mathbb{Z}$, it follows that neither row of $M_{Y''}$ is a zero row.  Lemma \ref{lemma-loops-2-by-2} then tells us that $Y''$ has loops if and only if $n''$ divides every entry in either the top or bottom row of $M_{Y''}$.  Otherwise, we have that $Y''$ (and hence $X$) is $3$-colorable unless one of the six exceptional cases in Theorem \ref{theorem-Heubergers} holds.  This gives us a total of eight cases to consider.  Of those, we write here only about the possibility that $n''$ divides both $-1-\frac32kt$ and $y_{22}+3k$.  The other seven cases can be managed using the same sort of techniques we've employed throughout this subsection; a full exposition can be found in \cite{Cervantes-Krebs}.

To recap, we now assume that $n''\mid -1-\frac32kt$ and $n''\mid y_{22}+3k$ and $t< 0$ and $k>0$.  So \[1+\frac32 kt\leq (3kt+1)y_{22}+\frac92 k^2t\leq -1-\frac32 kt.\]Solving for $y_{22}$ we find that \[\frac12-\frac32 k>y_{22}\geq -1-\frac32 k.\]Because $k,y_{22}\in\mathbb{Z}$, we have that $y_{22}=-\frac32 k+\epsilon$ where $\epsilon\in\{-1,-\frac12,0\}$.  We write here only about the case where $\epsilon=-1$.  The cases where $\epsilon=-\frac12$ or $\epsilon=0$ use the same sorts of techniques we've seen previously.  So assume $y_{22}=-\frac32 k-1$.  Then $n''=-3k(\frac12 +t)-1$ and $y_{22}+3k=\frac32 k-1>0$.  So from $n''\mid y_{22}+3k$ we have that \[-3k\left(\frac12 +t\right)-1\leq \frac32 k-1.\]Solving for $t$, we find that $t\leq -1$.  Because $t$ is a negative integer, this implies that $t=-1$.  Recall that $y_{21}=qy_{22}-1+\frac32 kt=(q+1)y_{22}$.  By Def. \ref{def-modified-Hermite-normal-form} we have that $\frac{y_{22}}{2}\leq y_{21}=(q+1)y_{22}\leq 0$.  Here we use that $y_{22}<0$.  Dividing by $y_{22}$ we get that $\frac12\geq q+1\geq 0$, so $q=-1$, because $q$ is an integer.  Then $y_{21}=0$.  From $t=q-\ell$ we get $\ell=0$, so $y_{21}=0$.  But then $X$ had loops, contrary to assumption.\end{proof}

The operations performed to put a $3\times 2$ matrix $M$ into modified Hermite normal form do not affect the gcd of the determinants of the $2\times 2$ minors of $M$.  Hence we have the following corollary to Theorem \ref{theorem-m-equals-3}, in analogy to Corollary \ref{corollary-det-divis-by-3}.

\begin{Cor}\label{corollary-minors-det-divis-by-3}
Suppose $X$ is a standardized abelian Cayley graph with an associated $3\times 2$ Heuberger matrix $M_X$.  If $X$ does not have loops, and if $3$ divides the determinant of every $2\times 2$ minor of $M_X$, then $X$ is $3$-colorable.
\end{Cor}

\begin{proof}If $M_X$ has a zero row, we may delete it without affecting the chromatic number, so in this case, the result follows from Cor. \ref{corollary-det-divis-by-3}.  If the columns of $M_X$ are linearly dependent over $\mathbb{Q}$, then after appropriate column operations we obtain a zero column, whereupon the result follows from \cite[Thm. 2.15]{Cervantes-Krebs-general-publilshed}.  Otherwise, as noted just before the statement of this corollary, we may assume that $M_X$ is in modified Hermite normal form.  But for each of the six types of matrices $M_X$ in the third statement in Theorem \ref{theorem-m-equals-3} for which $\chi(X)>3$, at least one $2\times 2$ minor has a determinant not divisible by $3$.  The result follows.\end{proof}

It would be interesting to know whether Cor. \ref{corollary-minors-det-divis-by-3} holds for matrices of arbitrary size.  We conjecture that it does.  Note that this is equivalent to the conjecture that if the rank of the reduction of $M_X$ modulo $3$ is $\leq 1$, then $\chi(X)\leq 3$.

\vspace{.05in}

Indeed, we remark that Thm. \ref{theorem-m-equals-3} can be recast entirely so that one can determine the chromatic number directly from an arbitrary Heuberger matrix $M$, not necessarily in modified Hermite normal form.  Namely, let $\alpha, \beta, \gamma$ be the absolute values of the determinants of the $2\times 2$ minors of $M$, such that $\alpha\leq \beta\leq \gamma$.  The six exceptional cases in Thm. \ref{theorem-m-equals-3} occur precisely when (i) one row of $M$ has relatively prime entries, and (ii) $\alpha>0$, and (iii) either $\{\alpha, \beta, \gamma\}=\{1,2,3k\}$ for some positive integer $k$, or else $\gamma=\alpha+\beta$ and $\alpha\not\equiv\beta\Mod{3}$.  The other cases (bipartite, loops, zero row, etc.) can easily be characterized directly from $M$.

\vspace{.05in}

We previously noted that for a standardized abelian Cayley graph $X$ with an associated $2\times 2$ Heuberger matrix $M_X$:

(*) If $X$ does not have loops, then $X$ fails to be $4$-colorable if and only if it contains $K_5$ as a subgraph, and it fails to be $3$-colorable if and only if it contains either $C_{13}(1,5)$ or a diamond lanyard as a subgraph.

Consequently, (*) holds for a $3\times 2$ matrix with a zero row, as the corresponding graph equals a box product of a doubly infinite path graph and a graph with a $2\times 2$ matrix.  Each of the six exceptional cases in the third statement in Theorem \ref{theorem-m-equals-3} contains a diamond lanyard as a subgraph, as we saw in the proofs of Lemmas \ref{lemma-first-column-pos-1-neg-1-pos-1}, \ref{lemma-L-shaped}, and \ref{lemma-I-on-top}.  Thus (*) holds also for every standardized abelian Cayley graph $X$ with an associated $3\times 2$ Heuberger matrix $M_X$.

\subsection{An algorithm to find the chromatic number for $1\times r, m\times 1, 2\times r$, or $3\times 2$ matrices}\label{subsection-algorithm}

In this subsection, we provide a ``quick-reference guide'' to the results of this section.  Specifically, we spell out a procedure to determine the chromatic number of a standardized abelian Cayley graph with a Heuberger matrix $M_X$ of size $1\times r, m\times 1, 2\times r$, or $3\times 2$, where $m$ and $r$ are positive integers.  This procedure can easily be converted into code or pseudocode.  Indeed, an implementation of this algorithm in {\it Mathematica} can be found at \cite{Cervantes-Krebs}.

$\bullet$ If $M_X$ is of size $m \times 1$, apply \cite[Theorem 2.15]{Cervantes-Krebs-general-publilshed}.

$\bullet$ If $M_X$ is of size $1 \times r$, apply Lemma \ref{lemma-m-is-1}.

$\bullet$ If $M_X$ is of size $2 \times 2$, apply column operations as per \cite[Lemma 2.6]{Cervantes-Krebs-general-publilshed} to produce a lower-triangular matrix.  Then apply Theorem \ref{theorem-m-equals-2}; if the last statement in that theorem holds, use Theorem \ref{theorem-Heubergers} to complete the final step in the computation.

$\bullet$ If $M_X$ is of size $2\times r$ for $r>2$, perform column operations as per \cite[Lemma 2.6]{Cervantes-Krebs-general-publilshed} to produce a zero column.  Delete that column, and iterate this procedure until you have a $2\times 2$ matrix.  Then use the procedure from the previous bullet point.

$\bullet$ If $M_X$ is of size $3 \times 2$, do the following.  If $M_X$ has a zero row, delete that row, then use the procedure for $2\times 2$ matrices to find the chromatic number of the graph with the resulting matrix.  If the columns of $M_X$ are linearly dependent over $\mathbb{Q}$, perform column operations as per \cite[Lemma 2.6]{Cervantes-Krebs-general-publilshed} to produce a zero column.  Delete that column, and then apply \cite[Theorem 2.15]{Cervantes-Krebs-general-publilshed}.  Otherwise, use row and column operations as per Lemma \ref{lemma-modified-Hermite-normal-form} to find an isomorphic graph $X'$ with a Heuberger matrix $M_{X'}$ in modified Hermite normal form.  Then apply Theorem \ref{theorem-m-equals-3}.

\section*{Acknowledgments}

The authors wish to thank Tim Harris for his careful read of an early version of this manuscript and for his many helpful suggestions.

\bibliographystyle{amsplain}
\bibliography{references}

\providecommand{\MR}[1]{}
\providecommand{\bysame}{\leavevmode\hbox to3em{\hrulefill}\thinspace}
\providecommand{\MR}{\relax\ifhmode\unskip\space\fi MR }
% \MRhref is called by the amsart/book/proc definition of \MR.
\providecommand{\MRhref}[2]{%
  \href{http://www.ams.org/mathscinet-getitem?mr=#1}{#2}
}
\providecommand{\href}[2]{#2}
\begin{thebibliography}{1}

\bibitem{Cervantes-Krebs-general-publilshed}
Jonathan Cervantes and Mike Krebs, \emph{Chromatic numbers of {C}ayley graphs
  of abelian groups: {A} matrix method}, Linear Algebra Appl. \textbf{676}
  (2023), 277--295. \MR{4622039}

\bibitem{Cervantes-Krebs-Zhu}
Jonathan Cervantes and Mike Krebs, \emph{Improved upper bounds for six-valent
  integer distance graph coloring periods}, Discrete Mathematics \textbf{348}
  (2025), no.~8, 114496.

\bibitem{Cervantes-Krebs}
{Cervantes, J. and Krebs, M.}, \emph{Authors' notes: Chromatic number of
  $\mathbb{Z}^n$ mod rank two subgroup},
  \url{https://www.calstatela.edu/research/mkrebs}.

\bibitem{Tim}
Timothy Harris, \emph{Standardized abelian {C}ayley graphs and their chromatic
  numbers in four dimensions}, M.{S}. thesis, California State University, Los
  Angeles, 2024.

\bibitem{Heuberger}
Clemens Heuberger, \emph{On planarity and colorability of circulant graphs},
  Discrete Mathematics \textbf{268} (2003), no.~1, 153–169.

\bibitem{MathOverflow}
{MathOverflow}, \emph{{C}hromatic numbers of infinite abelian {C}ayley graphs},
  \url{https://mathoverflow.net/questions/304715/chromatic-numbers-of-infinite-abelian-cayley-graphs},
  accessed {A}pr. 5, 2021.

\bibitem{Pommersheim}
James~E. Pommersheim, Tim~K. Marks, and Erica~L. Flapan, \emph{Number theory: A
  lively introduction with proofs, applications, and stories}, Wiley, Hoboken,
  New Jersey, 2010.

\bibitem{Soifer}
Alexander Soifer, \emph{The mathematical coloring book}, Springer, New York,
  2009, Mathematics of coloring and the colorful life of its creators, With
  forewords by Branko Gr\"{u}nbaum, Peter D. Johnson, Jr. and Cecil Rousseau.
  \MR{2458293}

\bibitem{Zhu}
Xuding Zhu, \emph{Circular chromatic number of distance graphs with distance
  sets of cardinality 3}, J. Graph Theory \textbf{41} (2002), no.~3, 195--207.
  \MR{1932297}

\end{thebibliography}

% \makeatletter\@input{xx.tex}\makeatother

% \end{document}

% \bibliographystyle{amsplain}
% \bibliography{references}

\end{document}